\documentclass[11pt, hidelinks]{article}
\usepackage[margin =1in]{geometry}
\usepackage{amssymb,amsthm,amsmath}
\usepackage{enumerate}
\usepackage{hyperref}
\usepackage{cite}
\usepackage[capitalize]{cleveref}
\usepackage{enumitem}
\usepackage{color}
\usepackage{cancel}
\usepackage{pgf}
\usepackage{svg}
\usepackage{pgfplots}
\usepackage{array}
\usepackage{import}
\usepackage{comment}
\usepackage[utf8]{inputenc} 
\usepackage[T1]{fontenc}    
\usepackage{hyperref}       
\usepackage{url}            
\usepackage{booktabs}       
\usepackage{amsfonts}       
\usepackage{mathrsfs}       
\usepackage{nicefrac}       
\usepackage{microtype}
\usepackage{multirow}
\usepackage{xfrac}
\usepackage{todonotes}
\usepackage{titlesec}
\usepackage{caption}
\usepackage{subcaption}
\usepackage{algorithm}
\usepackage{algpseudocode}
\usepackage{tabularx}

\usepackage{varwidth}

\titleformat{\subsubsection}[runin]
{\normalfont\normalsize\bfseries}{\thesubsubsection}{1em}{}

\usepackage{graphicx}

\usepackage[numbers,sort&compress]{natbib}
\usepackage{appendix}
\numberwithin{equation}{section}

\newcommand{\inclu}[0] {\ar@{^{(}->}}

\newcommand{\muhstar}{\frac{1}{2L_h}}

\newcommand{\prox}{\text{prox}}

\newcommand{\dom}[1]{\mathrm{dom }(#1)}

\newcommand{\argmin}{\operatornamewithlimits{argmin}}

\newcommand{\argmax}{\operatornamewithlimits{argmax}}

\newtheorem{thm}{Theorem}[section]
\newtheorem{theorem}[thm]{Theorem}
\newtheorem{definition}[thm]{Definition}
\newtheorem{proposition}[thm]{Proposition}

\newtheorem{lemma}[thm]{Lemma}

\newtheorem{corollary}[thm]{Corollary}

{\theoremstyle{plain}
}

\crefname{claim}{claim}{claims}
\Crefname{claim}{Claim}{Claims}
\crefname{lem}{lemma}{lemmas}
\Crefname{lem}{Lemma}{Lemmas}
\crefname{algorithm}{algorithm}{algorithms}
\Crefname{algorithm}{Algorithm}{Algorithms}

\newtheorem{remark}[thm]{Remark}

\theoremstyle{remark}

\newcommand{\eps}{\varepsilon}

\newcommand{\Lada}{L_{\eps,r}^{\mathtt{ADA}}}
\newcommand{\muada}{\mu_{\eps}^{\mathtt{ADA}}}

\newcommand{\Neps}{N_{\eps,r}}
\newcommand{\Peps}{P_{\eps,r}}
\newcommand{\LogBigO}{\Tilde{\mathcal{O}}}
\newcommand{\bigO}{\mathcal{O}}

\newcommand{\conv}[1]{\textrm{conv}({#1})}

\newcommand{\R}{\mathbb{R}}

\newcommand{\Lcal}{\mathcal{L}}
\newcommand{\hLcal}{{\mathcal{L}}}
\newcommand{\hQ}{{Q}}
\newcommand{\Z}{\mathcal{Z}}
\newcommand{\X}{\mathcal{X}}

\newcommand{\Holder}{H\"older }
\newcommand{\inner}[2]{\left\langle {#1}, {#2} \right\rangle}

\usepackage{mathtools}

\usepackage[T1]{fontenc}
\usepackage{lmodern}
\pgfplotsset{compat=1.18} 

\definecolor{blue}{rgb}{0,0,0}

\begin{document}

	\title{A Universally Optimal Primal-Dual Method\\ for Minimizing Heterogeneous Compositions}

	 \author{Aaron Zoll\footnote{Johns Hopkins University, Department of Applied Mathematics and Statistics, \url{azoll1@jhu.edu}} \qquad Benjamin Grimmer\footnote{Johns Hopkins University, Department of Applied Mathematics and Statistics, \url{grimmer@jhu.edu}}}

	\date{}
	\maketitle

\begin{abstract}
        This paper proposes a universal algorithm for convex minimization problems of the composite form $g_0(x)+h(g_1(x),\dots, g_m(x)) + u(x)$. We allow each $g_j$ to independently range from being nonsmooth Lipschitz to smooth, from convex to strongly convex, described by notions of H\"older continuous gradients and uniform convexity. Note that, although the objective is built from a heterogeneous combination of such structured components, it does not necessarily possess smoothness, Lipschitzness, or any favorable structure overall other than convexity. Regardless, we provide a universal optimal method in terms of oracle access to (sub)gradients of each $g_j$. The key insight enabling our optimal universal analysis {\color{blue} and a core technical contribution} is the construction of two new constants, the Approximate Dualized Aggregate smoothness and strong convexity, which combine the benefits of each heterogeneous structure into single quantities amenable to analysis.
        As a key application, fixing $h$ as the nonpositive indicator function, this model readily captures functionally constrained minimization $g_0(x)+u(x)$ subject to $g_j(x)\leq 0$. In particular, our algorithm and analysis are directly inspired by the smooth constrained minimization method of Zhang and Lan and consequently recover and generalize their accelerated guarantees.
        \end{abstract}

\maketitle

\section{Introduction}

This paper considers the design of scalable first-order methods for the following quite general class of convex optimization problems. Given closed, convex functions $g_j,\ u\colon \X \rightarrow \mathbb{R}\cup\{+\infty\}$ for $j=0, \dots, m$, a closed, convex, component-wise nondecreasing function $h\colon \mathbb{R}^m\rightarrow \mathbb{R}\cup\{+\infty\}$, and a closed, convex constraint set $\X\subseteq\mathbb{R}^n$, we consider the convex composite optimization problem
\begin{equation} \label{eq:main-composite-problem}
    p_\star = \min_{x\in \X} F(x) := g_0(x) + h(g_1(x),\dots,g_m(x)) +u(x)\ .
\end{equation}
We are particularly interested in {\it heterogeneous} settings where the components $ g_j$ forming the overall objective $F$ vary in their individual smoothness (or lack thereof) and convexity. The convex composite model~\eqref{eq:main-composite-problem} is well-studied and captures a range of standard optimization models:
\begin{itemize}
    \item {\it Minimization of Finite Summations.} As perhaps the most basic composite setup, minimization of finite sums $h(z)=\sum z_j$, where each $z_j=g_j(x)$ is one component of the objective, is widespread in machine learning and data science applications. The optimization of objective functions built from heterogeneous sums of smooth and nonsmooth components was recently considered by the fine-grained theory of~\citep{diakonikolas2024optimizationfinerscalebounded} and the bundle method theory of~\citep{liang2023}. Universal, optimal guarantees for the minimization of any sum of heterogeneously smooth components via Nesterov's universal fast gradient method~\citep{Nesterov_2014} were given by~\citep{grimmer2023optimal}.
    
    \item {\it Functionally Constrained Optimization.} Considering the composing function as the indicator function $h(z)=\iota_{z\leq 0}(z)$ for $z_j=g_j(x)$, this model recovers the standard notion of functionally constrained optimization. This setting has been studied significantly, with the recent smooth constrained optimization work of~\citep{zhang2022solving} being a particular motivation for this work. Constrained optimization handles a large class of problems with applications to machine learning, statistics, and signal processing~\citep{Lan.G,bregman,kavis2019unixgraduniversaladaptivealgorithm,Polak1997}.
  
    \item {\it Minimization of Finite Maximums.} Our model also captures minimizing finite maximums: $h(z)=\max_j z_j$ of several component functions $z_j =g_j(x)$~\citep{Polak1997}. For example, such objectives arise as a fundamental model in game theory, in robust optimization seeking good performance across many objectives~\citep{BenTal2009}, and in the radial optimization framework of~\citep{Grimmer2021-part1,Grimmer2021-part2}.
    
    \item {\it Smoothed Finite Maximum and Constrained Optimization} Finally, we provide two convex composite examples that address the previous two models in a smoothed setting. First, for applications minimizing the maximum of several functions $g_j(x)$, one can instead minimize an $\eta$-smoothing of the max function~\citep{smoothing_fom}: for some $\eta>0$, consider $h_\eta(z) = \eta \log(\sum_{j=1}^m \exp(z_j/\eta))$. As $\eta$ tends to zero, this converges to $\max_j z_j$ but becomes less smooth. Second, consider $h_\eta(z) = \sum_{j=1}^m \max\{z_j/\eta,0\}^2$, providing a smooth penalty for any constraint function violating nonpositivity. 
\end{itemize}

Here we address convex composite problems~\labelcref{eq:main-composite-problem}, assuming $u$, $h$, and $\X$ are reasonably simple (i.e., have a computable proximal/projection operator). Note that this captures all four of the above application settings. We allow $g_j$ to vary significantly in structure (i.e., ranging from nonsmooth Lipschitz to having Lipschitz gradients and from simple convexity to strong convexity). \cref{section:Holder-Compositions} and \cref{section:UC-restarting-comp} present the considered heterogeneous models of H\"older smoothness and uniform convexity formally, but we give the following definitions here. We say that $f$ is $(L,p)$-\Holder smooth for $p \in [0,1]$ if its gradient is \Holder continuous: \begin{equation}\label{eq:intro-Holder-def}\|\nabla f(x)-\nabla f(y)\| \leq L\|x-y\|^p \quad \forall x,y \in \dom{f} \ .\end{equation} As an immediate consequence of the fundamental theorem of calculus, \begin{equation}\label{eq:intro-Holder-consequence}
    f(y) \leq f(x)+\inner{\nabla f(x)}{y-x}+\frac{L}{p+1}\|y-x\|^{p+1} \quad \forall x,y \in \dom{f} \ .
\end{equation} Conversely, we say that $f$ is $(\mu,q)$-uniformly convex for $q \geq 1$ if\begin{equation}\label{eq:intro-UC-def}f(y) \geq f(x)+\inner{\nabla f(x)}{y-x}+\frac{\mu}{q+1}\|y-x\|^{q+1} \quad \forall x,y \in \dom{f} \ . \end{equation} We note that allowing each $g_j$ to satisfy these conditions with their own $(L_j,p_j)$ and $(\mu_j,q_j)$ does not guarantee $F$ possesses any of these favorable structures besides being simply convex. Despite this lack of centralized structure, this work presents a simple first-order method attaining optimal convergence guarantees, combining and leveraging whatever structure is present in each component.

Algorithms that can be applied optimally across a range of structurally different problem settings are known as {\it universal} methods. Universality is a key property for developing practical algorithms capable of being widely deployed in blackbox fashion.
For the case of minimizing a single function $f$ ranging in its \Holder smoothness, optimal universal methods were first pioneered by Lan~\citep{lan_uniformly_optimal, lan_bundle_optimal} and Nesterov~\citep{Nesterov_2014, nesterov_universal_stochastic}. {\color{blue}Further work on universal methods allowing for convex hybrid composite models~\citep{guigues2025universalsubgradientproximalbundle, liang2023},} heterogeneous summations~\citep{wang2022gradientcomplexityanalysisminimizing,grimmer2023optimal}, varied growth structures~\citep{Park_2022, ito2023parameterfreeconditionalgradientmethod}, constrained optimization~\citep{zhang2022solving,kavis2019unixgraduniversaladaptivealgorithm,deng2024uniformlyoptimalparameterfreefirstorder}, and {\color{blue} stochastic optimization}~\citep{Aybat2019, Rodomanov_2024} has followed since. To varying degrees, the above works developed algorithms that are ``mostly'' parameter-free, potentially relying on a target accuracy $\eps$, an upper bound on the diameter of $\X$, or similar universal problem constants. Without additional parameters, stopping criteria indicating when a target accuracy is reached are often unavailable. Hence, although the above methods apply universally, they vary in how parameter-free they are.

As a concrete example of a universal method, the Universal Fast Gradient Method (UFGM)~\citep{Nesterov_2014} can optimally minimize $F=g_0+u$, with the structure of $g_0$ ranging from smooth to nonsmooth. This setup is modeled by supposing $g_0$ is convex with $(L,p)$-H\"older continuous gradient, corresponding to Lipschitz gradients when $p=1$ and Lipschitz functions value when $p=0$. The UFGM, given target accuracy $\eps>0$ as an input, produces a point with at most $\eps$ objective gap in either of these settings and in every intermediate one, using at most 
\begin{equation}\label{eq:optimal-holderSmooth-rate}
   K_{SM}(\eps, L, p, \|x^0-x^\star\|) =   {\bigO}\left(\left(\frac{L}{\eps}\right)^{\frac{2}{1+3p}}\|x^0-x^\star\|^{\frac{2+2p}{1+3p}}\right)
\end{equation}
(sub)gradient oracle evaluations for $g_0$. The matching lower bounds cited in~\citep[page 26]{Nemirovski_restarting} establish that this rate is optimal for every $p\in[0,1]$. Given additional structure, like $(\mu,q)$-uniform convexity of $g_0$, a universal restarting scheme like~\citep{grimmerrestarting} can be applied to achieve the optimal, faster rates of
\begin{align}\label{eq:optimal-holderSmoothUniformlyConvex-rate} 
\begin{split}
K_{UC}(\eps, L, p, \mu, q) = \begin{cases}{\bigO}\left(\left(\frac{L^{1+q}}{\mu^{1+p}\eps^{q-p}}\right)^{\frac{2}{(1+3p)(1+q)}}\right) &\quad \text{if } q>p \\
    {\bigO}\left(\left(\frac{L^{1+q}}{\mu^{1+p}}\right)^{\frac{2}{(1+3p)(1+q)}}\log\left(\frac{F(x^0)-F(x^\star)}{\eps}\right)\right)  &\quad \text{if }q=p
    \end{cases}
    \end{split} 
\end{align}
(sub)gradient oracle evaluations {\color{blue} with respect to $g_0$}, up to logarithmic factors.

This work aims to develop a universal method for the composite setting~\eqref{eq:main-composite-problem}, allowing heterogeneity in the H\"older smoothness and uniform convexity of each $g_j$, capturing and generalizing the settings of the above universal methods. Our proposed Universal Fast Composite Method (UFCM) is formally defined in \Cref{alg:FCM}. When restarting is included, we denote it by R-UFCM, defined in \Cref{alg:R-FCM}. Our method is not parameter-free, depending on the following three main parameters: a target accuracy $\eps>0$, an Approximate Dualized Aggregate smoothness $\Lada$ capturing the combined effect of any upper bounding curvature present among the composition, and finally, an Approximate Dualized Aggregate convexity $\muada$ capturing the combined effect of any lower bounding curvature. {\color{blue} The invention of these unifying constants, abstracting and simplifying any complex dependence on individual components' H\"older smoothness and uniform convexity constants and exponents, is key to our algorithm's success.}
We formally define the latter parameters in \labelcref{eq:Lada-definition-general} and \labelcref{eq:muada-def-general}.

{\color{blue} We note that one may tradeoff knowledge of $\Lada$ for knowledge of bounds on the initial distances to optimal, $D_x$ and $D_\lambda$, without affecting big-$\bigO$ oracle complexities. Further discussion is presented in \cref{rmk:tradeoff-lada-dist}. The design of entirely parameter-free methods, avoiding knowledge of these aggregate constants and distance bounds, is left as an important future direction. The parameter-free techniques of~\citep{Nesterov_2014, adabbadaptivebarzilaiborweinmethod,  li2024simpleuniformlyoptimalmethod} may be useful to this extent.}

Measuring the convergence of a method requires a suitable notion of solution quality. Often, iterative methods seek to produce a solution $x^t$ with a bounded objective gap:
\begin{equation} \label{eq:bounded-suboptimality}
    F(x^t) - p_\star \leq \eps \ . 
\end{equation}
However, for general composite problems~\eqref{eq:main-composite-problem}, since $F$ is allowed to take infinite value arbitrarily near a minimizer (an important attribute for modeling constrained optimization as discussed above), our iterative schemes for minimizing $F$ do not directly provide a solution $x^t$ with bounded suboptimality. Instead, we identify $(\eps,r)$-optimal $x^t$ defined as there existing $\hat{g}\in\mathbb{R}^m$ and a subgradient $\hat\lambda \in \partial h(\hat{g})$ satisfying
\begin{equation}
        \begin{cases}  g_0(x^t) + h(\hat g) + \langle \hat \lambda, g(x^t) - \hat g \rangle +u(x^t) - p_\star &\leq  \eps  \ , \\
        r\|g(x^t)-\hat g\| &\leq  \eps \ .
        \end{cases}\label{def:optimality-conditions}
    \end{equation}
Here, {\color{blue} $\hat{g}$ informally serves as a perturbed projection of $g(x^t)$ onto the domain of $\partial h$ and} $r$ is a hyperparameter that one can fix proportional to $\sqrt{\eps}$ to obtain simply an ``$\eps$-optimal'' solution where $\|g(x^t)-\hat g\|^2 \lesssim \eps$ (\cref{lemma:Q-small->eps-opt} introduces $r$ and discusses its meaning as a radius for dual multipliers). 

This condition states $x^t$ nearly attains the optimal objective value when the outer composition function $h$ is linearized via a subgradient $\hat \lambda$ taken at a nearby $\hat g$. For example, in the context of constrained minimization where $h$ is an indicator function for the nonpositive orthant, $\hat \lambda$ is precisely a {\color{blue} nonnegative} vector of Lagrange multipliers, making the above conditions correspond to the approximate attainment of the KKT conditions. {\color{blue} In this case, $\hat{g}$ is a perturbed projection of $g(x^t)$ onto the nonpositive orthant. By construction, $\hat \lambda$ and $\hat g$ are orthogonal, and the conditions for an $(\eps,r)$-optimal solution correspond to $$\begin{cases}
    g_0(x^t) + \inner{\hat{\lambda}}{g(x^t)}+u(x^t)-p_\star & \leq \eps \ , \\
    r\|g(x^t)-\hat{g}\| & \leq \eps \ ,\\
    \hat{g}  \leq 0, \
    \hat{\lambda} \geq 0 \ .
 \end{cases}$$
 The first condition states that $x^t$ approximately minimizes the Lagrangian at $\hat{\lambda}$. Approximate primal feasibility follows from the second and third conditions establishing $\mathrm{dist}(g(x^t), \R^m_{-}) \leq \eps/r$. Dual feasibility follows from the nonnegativity of $\hat\lambda$. Finally, approximate complementary slackness follows from the orthogonality of $\hat{\lambda}$ and $\hat{g}$ as $|\langle \hat{\lambda}, g(x^t)\rangle| = |\langle \hat{\lambda}, g(x^t)-\hat{g}\rangle| \leq \|\hat{\lambda}\|\eps/r$.
    }

    As a second example, when $h$ is a linear function (e.g. when directly minimizing a sum of component functions), one has $h(g(x^t)) = h(\hat g) + \langle \hat \lambda, g(x^t) - \hat g \rangle$ and so~\eqref{def:optimality-conditions} reduces to the classic bounded suboptimality measure~\eqref{eq:bounded-suboptimality}. {\color{blue} In this case, $\hat{g}=g(x^t)$. Further discussion on the roles and values of $\hat{g}$ and $\hat{\lambda}$ is in \cref{subsec:Q-for-compositions}}

 {\color{blue} Note in our developed algorithms, both $\hat{g}$ and $\hat{\lambda}$ are not explicitly constructed, and are generally inaccessible computationally. Hence, although our analysis guarantees such values exist certifying approximate optimality, we cannot certify at runtime when the iterate becomes $(\eps,r)$-optimal. This limitation cannot be improved: When $h$ is linear, our optimality condition~\eqref{def:optimality-conditions} reduces to suboptimality $F(x^t) - p_\star \leq \epsilon$, which cannot be certified without knowledge of $p_\star$ or additional global structure.}

\subsection{Our Contributions}\label{section:contributions} 
This work develops a universal primal-dual method for heterogeneous compositions~\eqref{eq:main-composite-problem} {\color{blue} with optimal first-order complexity with respect to components $g_j$}. Our proposed UFCM and its restarted variant R-UFCM leverage the sliding technique of~\citep{lan2014gradientslidingcompositeoptimization} and the ``Q-analysis'' technique of~\citep{zhang2022solving}, {\color{blue} originally} designed for smooth constrained optimization. As a key contribution to this end, we propose new notions of the Approximate Dualized Aggregate smoothness constant $\Lada$ and the Approximate Dualized Aggregate convexity constant $\muada$. {\color{blue} These two constants provide a new unifying technical tool for the analysis of heterogeneous optimization that may be of independent interest.} From these, we prove the oracle complexities outlined in Table~\ref{tab:composite-heterogeneous-rates}. For example, in the simple setting of minimizing $g_0(x)+u(x)$, these rates recover the optimal {\color{blue} suboptimality convergence} rates of~\eqref{eq:optimal-holderSmooth-rate} and~\eqref{eq:optimal-holderSmoothUniformlyConvex-rate}.

\begin{table}[t]
\centering

\newcommand{\cellstrut}{\rule{0pt}{6.3ex}\rule[-4.0ex]{0pt}{0pt}}
\newcommand{\smallcellstrut}{\rule{0pt}{3.2ex}\rule[-2.0ex]{0pt}{0pt}} 
\newcommand{\mucellstrut}{\rule{0pt}{5.2ex}\rule[-4.0ex]{0pt}{0pt}} 

\newcolumntype{M}{>{\centering\arraybackslash}m{4.4cm}}
\newcolumntype{N}{>{\centering\arraybackslash}m{4cm}}
\newcolumntype{S}{>{\centering\arraybackslash}m{2.8cm}}
\newcolumntype{X}{>{\centering\arraybackslash}m{1.8cm}}

\begin{tabular}{|X|S|M|N|}
\hline
\smallcellstrut \multirow{2}{0cm}{} &
\smallcellstrut \multirow{2}{2.8cm}{%
  \parbox[c][1.1cm][c]{2.8cm}{\centering\textit{\textbf{First-Order Oracle Calls to Components} $g$}}%
}&
\multicolumn{2}{c|}{%
  \parbox[c][1.8cm][c]{\dimexpr4cm+4.4cm\relax}{\centering
  \textit{\textbf{Proximal Oracle Calls to} $h$ \textbf{and} $u$}}%
} \\
\cline{3-4}
& &
\smallcellstrut ${ \scriptstyle {L_h \ > \ {D_\lambda^2}/{\eps}}}$ &
\smallcellstrut  ${ \scriptstyle L_h \ \leq \ {D_\lambda^2}/{\eps}} $\\
\hline
 \mucellstrut ${\scriptstyle \mu_\eps^{\mathtt{ADA}}\ < \ \eps/D_x^2}$  &
\cellstrut $\displaystyle\sqrt{\frac{\Lada D_x^2}{\eps}}$  &
\cellstrut $\displaystyle\sqrt{\frac{\Lada D_x^2}{\eps}}+\frac{MD_xD_\lambda}{\eps}$ &
\cellstrut $\displaystyle\sqrt{\frac{(\Lada + M^2L_h)D_x^2}{\eps}}$\\
\hline
\mucellstrut ${\scriptstyle \mu_\eps^{\mathtt{ADA}} \ \geq \ \eps/D_x^2}$  &
\cellstrut $\displaystyle\sqrt{\frac{\Lada}{\muada}}\log\left(\frac{1}{\eps}\right)$ &
\cellstrut $\displaystyle\sqrt{\frac{\Lada}{\muada}}\log\left(\frac{1}{\eps}\right)+\frac{MD_\lambda}{\sqrt{\muada \eps}}$ &
\cellstrut $\displaystyle \sqrt{\frac{\Lada+M^2L_h}{\muada}}\log\left(\frac{1}{\eps}\right)$\\
\hline
\end{tabular}

  \caption{Oracle complexities in terms of universal parameters $\Lada$ and  $\muada$, proven in Theorem~\ref{thm:fully-heteogeneous-component-guarantees}, up to constants and additive logarithmic terms in $\eps$. Here, $D_x$ and $D_\lambda$ denote bounds on the initial primal and dual distances to optimality, and $M$ denotes a local Lipschitz constant. 
    } \label{tab:composite-heterogeneous-rates}
\end{table}

\noindent For ease of exposition, we develop our convergence theory incrementally through three main theorems:\\
--{Theorem~\ref{thm:smooth-composite-convergence}} establishes an $\bigO(1/\sqrt{\eps})$ rate towards $\eps$-optimality when each $g_j$ is smooth and convex. Hence, smooth composite optimization is nearly as easy as unconstrained smooth optimization.\\ 
--{Theorem~\ref{thm:heterogeneous-composite-convergence}} generalizes this analysis to establish optimal rates when each $g_j$ is convex with varying H\"older continuous gradient~\eqref{eq:intro-Holder-def}, recovering~\eqref{eq:optimal-holderSmooth-rate} as a special case.\\
--{Theorem~\ref{thm:fully-heteogeneous-component-guarantees} }finally leverages standard restarting techniques to establish optimal rates when the components additionally possess varying uniform convexity~\eqref{eq:intro-UC-def}, recovering~\eqref{eq:optimal-holderSmoothUniformlyConvex-rate} as a special case.

{\color{blue} \begin{remark}
    Note that these rates are only optimal for the first-order complexity with respect to the components $g_j$, not necessarily the proximal oracle complexity. The work of~\citep{efficient_prox_paper} shows the latter can be improved to $\bigO\left(1/\eps\right)$ when $g_0$ is nonsmooth and Lipschitz whereas our method requires $\bigO\left(1/\eps^2\right)$ proximal evaluations in such nonsmooth settings.
\end{remark}}
{\color{blue}

\subsection{Example of our Universal Constants $\Lada$ and $\muada$ and  an Application of Convergence Rates.}
Our ability to provide universal guarantees across heterogeneous problem settings is primarily enabled by the design of our Approximate Dualized Aggregate smoothness $\Lada$ and strong convexity $\muada$ constants. Although we defer formal definitions of these to \labelcref{eq:Lada-definition-general} and \labelcref{eq:muada-def-general}, here we briefly discuss their essential properties and consequences.
These constants are ``approximate'' in that they depend on the target accuracy $\epsilon$, ``dualized'' in that they depend on associated optimal dual multipliers $\lambda_j^\star$, and ``aggregate'' in that they combine these dependencies and every problem parameter $(L_j,p_j)$, $(\mu_j,q_j)$, etc.~into a single constant. From these constants, we find that the traditional smooth and smooth, strongly convex rates $\bigO\left(\sqrt{\frac{\Lada D_x^2}{\eps}}\right)$ and $\bigO\left(\sqrt{\frac{\Lada}{\muada}} \log\left(\frac{1}{\eps}\right)\right)$ hold for generic heterogeneous composite settings.

These unifying constants are graceful in their dependence on dual multipliers $\lambda_j^\star$: the dependence on the $j$th component's $(L_j,p_j)$ and $(\mu_j,q_j)$ vanishes as $\lambda_j^\star$ tends to zero. In constrained optimization, $\lambda_j^\star=0$ corresponds to the constraint being inactive at the optimal solution. Hence, inactive constraints play a vanishing role in our convergence rates (as one would hope). As a more concrete example, consider minimizing a finite maximum $h(g_1(x), g_2(x))$ with $h=\max\{z_1,z_2\}$ of an $L$-smooth function $g_1$ and an $M$-Lipschitz nonsmooth function $g_2$. Here, the optimal dual multiplier $\lambda^\star \in [0,1]$ describes the activity of each component at the minimizer, $\lambda^\star=0$ if the problem reduces to minimizing the smooth component, $\lambda^\star=1$ if the problem reduces to minimizing the Lipschitz component, and $\lambda^\star\in (0,1)$ if both are active. \cref{cor:Lada-general-implicit-formula} shows that our gradient complexity guarantees for such problems simply decompose into the sum of each component's complexity separately, weighted by its dual multiplier plus $r$,
$$ \bigO\left(\sqrt{\frac{(1-\lambda^\star + r)L D_x^2}{\eps}}+\frac{((\lambda^\star +r) M)^2 D_x^2}{\eps^2}\right) \ . $$
Selecting $r=O(\epsilon^{3/4})$, this bound transitions from the optimal accelerated smooth rate $O(1/\sqrt{\epsilon})$to the optimal nonsmooth rate $O(1/\epsilon^2)$ gracefully as $\lambda^\star\in [0,1]$ varies.

This recovery of the optimal rates~\eqref{eq:optimal-holderSmooth-rate} when there is a single active component establishes near optimality of our guarantees with respect to first-order oracle evaluations of the component functions $g_j$. Any improvement in our dependencies in $\bigO\left(\sqrt{\frac{\Lada D_x^2}{\eps}}\right)$ and $\bigO\left(\sqrt{\frac{\Lada}{\muada}} \log\left(\frac{1}{\eps}\right)\right)$ beyond a log term would violate the lower bounds stated by~\citep{Nemirovski_restarting}.

}

\paragraph{Outline.}
 Section~\ref{Heterogeneous-Comp-Prelim} introduces preliminaries as well as the sliding technique and ``Q-analysis'' discussed in~\citep{zhang2022solving} for solving constrained optimization. Section~\ref{section:Smooth-Composite} extends this method to smooth composite optimization, proving optimal guarantees. Section~\ref{section:Holder-Compositions} generalizes to functions with H\"older continuous gradient. Finally, Section~\ref{section:UC-restarting-comp} generalizes to allow heterogeneous levels of uniform convexity. 

\section{Preliminaries}\label{Heterogeneous-Comp-Prelim}

We define our notation to align with~\citep{zhang2022solving}'s prior work in constrained optimization. 
First, without loss of generality we set $g_0(x)=0$ as one can consider instead minimizing $0+\hat h(g_0(x),\dots g_m(x)) +u(x)$ with $\hat h(z_0,z_1,\dots z_m) = z_0 + h(z_1,\dots z_m)$. Hence, it suffices to consider problems of the form
\begin{equation}\label{eq:WLOG-main-problem}
    \min_{x \in \X} F(x) := h(g_1(x), \dots , g_m(x)) +u(x) \ ,
\end{equation}
defined by a closed, convex set $\X \subseteq \R^n$ and the following closed, convex functions: regularizing function $u \colon \X \to \R \cup \{+\infty\}$, composing function $h \colon \R^m \to \R \cup \{+\infty\}$, and component functions $g_j \colon \X \to \R $. Below, we describe the additional structures assumed on each function.

\paragraph{Assumed Structure of Objective Components $g, h,u$.}
We assume each $g_j$ is locally Lipschitz with some form of bounds on its curvature. We take each $g_j$ to be $L_j$-smooth (i.e., $\nabla g_j$ is $L_j$-Lipschitz) in \cref{section:Smooth-Composite} to set up the algorithmic framework and convergence results. In Sections~\ref{section:Holder-Compositions} and~\ref{section:UC-restarting-comp}, we allow the components to have varying levels of smoothness and varying levels of convexity, as defined in \labelcref{eq:intro-Holder-def} and \labelcref{eq:intro-UC-def}. Whatever structure is present in these components only arises in our convergence theory through the unifying parameters $\Lada$ and $\muada$, which aggregate any structures available, enabling our universal method and analysis.
We assume $\X$, $u$, and $h$ are sufficiently simple that their proximal operators can be evaluated, defined for any parameter $\tau>0$ as
\begin{align}
    \prox_{u,\tau}(x)&:=\argmin_{y \in \X} u(y)+\frac{\tau}{2} \|y-x\|^2 \label{eq:prox-def}, \\
    \prox_{h,\tau}(x)&:=\argmin_{y \in \R^m} h(y)+\frac{\tau}{2} \|y-x\|^2 \label{eq:h-prox-def}
\end{align}
respectively.

The algorithms designed herein are primal-dual, leveraging the convex (Fenchel) conjugates~\citep{Fenchel_1949} of $h$ and each $g_j$. For any convex function $f \colon \R^n \to \R \cup\{+\infty\}$, we denote its conjugate as
\begin{equation}
    f^*(s)=\sup_{x \in \R^n} \{\inner{s}{x}-f(x)\} \ . \label{eq:Fenchel-conjugate}
\end{equation}
Note Moreau's decomposition~\citep[Theorem 6.45]{bregman} shows $\prox_{h^*,\tau}(x)=x- \prox_{h,1/\tau}(\tau x)/\tau$ and so the assumed oracle access to $\prox_h$ via~\eqref{eq:h-prox-def} further provides access to $\prox_{h^*}$.

Finally, we assume $h$ is component-wise nondecreasing, which suffices to ensure the overall objective $F$ is convex. The following pair of standard lemmas formalize the resulting structures.
\begin{lemma}~\citep[Theorem 5.1]{rockafeller} If $h \colon \R^m \to \R \cup \{+\infty\}$ is convex and component-wise nondecreasing and $g \colon \R^n \to \R^m$ is component-wise convex, then $c(x) := h(g(x))\colon \R^n \to \R$ is convex.
\end{lemma}
\begin{lemma}\label{lemma:lambda nonneg} If $h \colon \R^m \to \R \cup \{+\infty\}$ is convex and component-wise nondecreasing, then $\dom{h^*} \subseteq \R_+^m$. 
\end{lemma}
\begin{proof}
    Since $h$ is component-wise nondecreasing and convex, at any $x \in \dom{h}$, \begin{align*}
        h'(x; -e_j) \leq 0, \ \forall j
       \Longleftrightarrow  \sup_{s \in \partial h(x)} \inner{s}{-e_j} \leq 0, \ \forall j
        \Longleftrightarrow  \forall s \in \partial h(x), \ s \geq 0 \ 
         \Longleftrightarrow \partial h(x) \subseteq \R_+^m\ .
    \end{align*}
    Then, for any $s \in \mathrm{ri}(\dom{h^*})$, there exists $x \in \partial h^*(s)$, and thus $s \in \partial h(x) \subseteq \R^m_+$. 
\end{proof} 

\paragraph{Lagrangian Reformulations.} We can now define a Lagrangian function essential to our algorithm and its analysis. Recalling $f=f^{**}$ for any closed, convex, and proper function~\citep[Corollary 12.2.1]{rockafeller}, one has that
$$h(g(x))=\sup_{\lambda \in \Lambda} \inner{\lambda}{g(x)}-h^*(\lambda), \quad \text{where } \Lambda := \dom{h^*}\ .$$
The \textit{Standard Lagrangian} reformulation follows as
\begin{equation}\label{eq:simple-lagrangian}
    \inf_{x \in \X} h(g(x))+u(x) = \inf_{x \in \X}\sup_{\lambda \in \Lambda} \left\{    \inner{\lambda}{g(x)} - h^*(\lambda)+u(x)\right\} \ .
\end{equation}
Furthermore, since each $\lambda \in \R_+^m$ (see \cref{lemma:lambda nonneg}), one can dualize each component function $g_j$, obtaining the equivalent \textit{Extended Lagrangian} reformulation, which our analysis will utilize
\begin{equation}\label{eq:modified-lagrangian}
    \inf_{x \in \X}  h(g(x))+u(x) = \inf_{x \in \X} \sup_{(\lambda,  \nu) \in \Lambda \times V} \left\{\Lcal(x;\lambda,\nu) := \inner{\lambda}{\nu x - g^*(\nu)}-h^*(\lambda)+u(x)\right\} \ ,
\end{equation}
{\color{blue}where $V := \dom{g^*}$. } {\color{blue} Note that $\Lcal(x; \lambda, \nu)$ is convex in $x$ and block-wise concave in $\lambda$ and $\nu$.}
 
Define $\Z := \X \times \Lambda  \times V$ for primal variables $x \in \X$, dual variables $\lambda \in \Lambda = \dom{h^*} \subseteq \R_+^m$, and conjugate variables $\nu \in V = \dom{g^*} \subseteq \R^{m \times n}$\ . Let $Z^\star$ denote the set of saddle points of~\labelcref{eq:modified-lagrangian}, which we assume throughout is nonempty.  {\color{blue} In the case where $h$ is linear, this assumption is equivalent to the existence of a minimizer. In the case where $h = \iota_{\leq 0}$, the functionally constrained setting, this assumption is equivalent to strong duality holding with primal and dual attainment.} Note any such $z^\star \in Z^\star$ must have $0\in\partial_\lambda \Lcal(x^\star;\lambda^\star,\nu^\star)$ and consequently $\lambda^\star \in \partial h(g(x^\star))$.

As a common generalization of the Euclidean distance, for any convex reference function $g \colon \X \to \R$, we define the associated Bregman divergence as
\begin{equation}\label{eq:Bregman-def}
    U_g(x;\hat{x}):=g(x)-g(\hat{x})-\inner{g'(\hat{x})}{x-\hat{x}} 
\end{equation}
for some $g'(\hat{x}) \in \partial g(\hat{x})$. If $g$ is vector-valued, we extend the definition above by $g=(g_1, ..., g_m)$ and $U_{g}=(U_{g_1}, ..., U_{g_m})$. That is, $U_g$ is vector-valued with each component being the Bregman divergence of the corresponding component of $g$.

\subsection{Key Techniques from Prior Works}\label{subsec:related-work}
Our results rely on four technical tools developed over the last decade that we bring together to handle various facets of the general problem~\eqref{eq:main-composite-problem}.  We introduce these formally below. In short, Lan's sliding technique~\citep{lan2014gradientslidingcompositeoptimization} allows us to decompose the complexity concerning proximal steps on $u$ and $h$ from that of gradient calls to $g$; the Q function analysis of~\citep{zhang2022solving} provides the primal-dual framework from which UFCM is built; the technique to universally analyze \Holder smooth functions from~\citep{Nesterov_2014} allows our results to generalize beyond smooth optimization; restarted methods allow us to generalize our results further to benefit from any uniform convexity present among its components. \\

\noindent {\it Sliding Gradient Methods.} The sliding technique introduced by Lan~\citep{lan2014gradientslidingcompositeoptimization, lan_conditional_gradient} iteratively and approximately solves subproblems associated with accelerated proximal gradient methods. This approach was first developed to handle objectives $g_0+u$, with $g_0$ smooth and $u$ nonsmooth but with readily available subgradients. The sliding gradient method allows the number of first-order oracle calls to $g_0$ and $u$ to be decomposed, often significantly reducing the number of calls needed to $g_0$.
In the context of our considered method, a central step of our method requires a proximal step on a certain minimax optimization subproblem involving $u$ and $h^*$. Sliding performs this step approximately, decomposing computations related to $\nabla g$, $\prox_{u,\tau}$, and $\prox_{h^*,\tau}$. \\

\noindent {\it \texorpdfstring{$Q$}{Lg} Function Framework for Constrained Optimization.}
The novel work of~\citep{zhang2022solving} considered the problem of minimizing $g_0+u$ subject to inequality constraints $g_j(x)\leq 0$, corresponding in our model to minimizing  $F=g_0(x)+\iota_{z \leq 0}(g_1, \dots, g_m) + u(x)$. The key step therein is designing algorithms generating iterates $z^t$ driving an associated ``gap function'' providing a measure of optimality on the extended Lagrangian reformulation to zero\footnote{Here $\pi$ is dual multiplier corresponding to the function $g_0$, always equal to $\pi^t := \nabla g_0(x^t)$. We omit this variable from the formulation considered throughout this work as without loss of generality, we set $g_0=0$.}:
$$Q(z^t, z) := \Lcal(x^t; \lambda, \nu, \pi)-\Lcal(x; \lambda^t, \nu^t, \pi^t) \ .$$
Their proposed accelerated method for smooth constrained optimization works by optimizing the Q function separately with each block of variables. Their updates concerning $\pi$ and $\nu$ amount to computing gradients of $g_0$ and $g_j$. Their updates for $x$ and $\lambda$ correspond to solving a quadratic program, which a sliding technique is able to decompose.

Our theory recovers these results of Zhang and Lan, improving their guarantees in settings with strongly convex constraints, enabling it to apply universally to compositions (not just constrained optimization) and to problems with H\"older smooth and/or uniformly convex components. Section~\ref{subsec:Q-for-compositions} formally develops our generalization of their Q function framework. After developing our convergence theory, Section~\ref{section:application-to-func-constrained} provides a detailed comparison of results.\\

\noindent {\it Universal Methods for Minimization with H\"older Continuous Gradient.}

Nesterov's universal fast gradient method~\citep{Nesterov_2014} provided a generalization of Nesterov's classic fast gradient method~\citep{NAG} capable of minimizing any $(L,p)$-H\"older smooth objective. The key technical insight {\color{blue} originating from~\citep{Convex Opti Inexact Oracle}} that enables this method is a lemma establishing an approximate smoothness result for any such function, meaning the standard quadratic upper bound derived for functions with Lipschitz gradient holds for functions with H\"older gradient up to an additive constant. A variant of this lemma showcasing a standard cocoercivity inequality also generalizes at the cost of an additive constant, derived by~\citep{li2024simpleuniformlyoptimalmethod}.

\begin{lemma}[Lemma 1,~\citet{Nesterov_2014}]\label{lemma:Nesterov-universal-quad}
    For any tolerance $\delta>0$ and $(L, p)$-\Holder smooth function $f\colon \X \to \R$ with  $L_\delta \geq \left[\frac{1-p}{1+p}\frac{1}{\delta}\right]^{\frac{1-p}{1+p}}L^{\frac{2}{1+p}}$, 
     \begin{equation}\label{eq:holder-quad-upper}
        f(y) \leq f(x)+ \inner{\nabla f(x)}{y-x} +\frac{L_\delta}{2}\|y-x\|^2+\frac{\delta}{2}, \ \forall x,y \in \dom{f}\ .
    \end{equation}
\end{lemma}
\begin{lemma}[Lemma 1,~\citet{li2024simpleuniformlyoptimalmethod}]\label{lemma:cocoercive-Holder}
   In the same setting as \cref{lemma:Nesterov-universal-quad},  \begin{equation}\label{eq:holder-cocoercive-bound}
        f(y) \geq f(x) + \inner{\nabla f(x)}{y-x} + \frac{1}{2L_\delta}\|\nabla f(x)-\nabla f(y)\|^2 - \frac{\delta}{2}, \  \forall x,y \in \dom{f} \ .
    \end{equation}
\end{lemma}
These lemmas facilitate our generalization in Section~\ref{section:Holder-Compositions} from smooth components to heterogeneously H\"older smooth components. Our Approximate Dualized Aggregate smoothness constant $\Lada$ is a further generalization of the approximate smoothness constants $L_\delta$ seen above. Namely, $\Lada$ further aggregates the \Holder smoothness of each component $g_j$, weighted approximately by the corresponding optimal dual multiplier $\lambda_j^\star$.\\

\noindent {\it Restarting Gradient Methods.} \label{subsubsec:restarting}
Algorithmic restarting, dating back to at least~\citep {Nemirovski_restarting}, can be shown to accelerate the convergence rate of first-order methods. The more recent works~\citep{Yang2018,Roulet2020,grimmerrestarting} established improved convergence guarantees given strong or uniform convexity or any general H\"olderian growth. The analysis of such schemes tends to rely on ensuring a reduction, often a contraction, in the distance to optimal occurs at each restart. Such schemes have found particular success in primal-dual algorithm design for linear and quadratic programming~\citep{lu_LP_sharp, lu2024practicaloptimalfirstordermethod}.

In our analysis, two distances to optimal are traced based on the distance from primal iterates $x^t$ to $x^\star$ and the distance from the dual iterates $\lambda^t$ to $\lambda^\star$. Given any uniform convexity among the components $g_j$, our Approximate Dualized Aggregate convexity $\muada$ describes the improvement in convergence gained from restarting the primal iterate sequence. Given any smoothness $L_h$ in the composing function $h$, improved convergence follows from restarting the dual iterate sequence. The relative sizes of $\muada$ and $L_h$ determine our various rates previously claimed in Table~\ref{tab:composite-heterogeneous-rates}.

\section{Minimization of Compositions with Smooth Components}\label{section:Smooth-Composite}

For ease of exposition, in this section, we first develop our main algorithm UFCM, assuming each component function $g_j$ is $L_j$-smooth and convex. The following two sections provide extensions to benefit from any H\"older smoothness and uniform convexity present in each $g_j$ and any smoothness present in $h$. Section~\ref{subsec:Q-for-compositions} formalizes the Q analysis framework for our composite optimization context, and Section~\ref{subsec:Lada} introduces our unifying Approximate Dualized Aggregate smoothness parameter $\Lada$. Then, Section~\ref{subsec:UFCM} presents our first convergence guarantee, only requiring the Approximate Dualized Aggregate smoothness $\Lada$ (or any upper bound) as input. Finally, Section~\ref{subsec:smooth-analysis} provides the key steps in our analysis, deferring any reasoning {\color{blue} directly} generalizing the constrained optimization analysis of~\citep{zhang2022solving} to the appendix.

\subsection{\texorpdfstring{$Q$}{Lg} Function Framework for Composite Optimization} \label{subsec:Q-for-compositions}
We can now introduce our generalization of the Q analysis framework of~\citep{zhang2022solving} that drives this paper. Based on the extended Lagrangian~\labelcref{eq:modified-lagrangian}, we define an analogous gap function.
\begin{definition}
    Given functions $g,h,u$ defining an instance of~\eqref{eq:WLOG-main-problem}, the gap function is defined as
    \begin{equation}\label{eq:gap-function}
        \hQ(z, \hat{z}) := \hLcal(x ; \hat{\lambda}, \hat{\nu})-\Lcal(\hat{x};\lambda, \nu) \ .
    \end{equation}
\end{definition}
\noindent Fixing $h(\cdot)$ as the indicator for the nonpositive orthant recovers their definition.

For the sake of our analysis, we fix an arbitrary saddle point $z^\star:=(x^\star; \lambda^\star,\nu^\star)$ with $\nu^\star := \nabla g(x^\star)$. Note $\hLcal(x^\star; \lambda, \nu) \leq \hLcal(x^\star ; \lambda^\star,  \nu^\star) \leq \hLcal(x; \lambda^\star, \nu^\star)$. Hence, $Q(z,z^\star) \geq 0$ for all  $z \in \Z$, making convergence of $Q(z^t,z^\star)$ a potential measure of solution quality. Our analysis considers a slight modification, allowing perturbations of $\lambda^\star$ and $\nu^\star$, giving a condition that implies $z^t$ is an $(\eps,r)$-optimal solution~\labelcref{def:optimality-conditions} for our original composite problem.
To this end, we restrict to considering $\lambda$ within a fixed distance $r$ of $\lambda^\star$ and in the dual domain $\Lambda = \dom{h^*}$, denoted
\begin{equation}\label{eq:reference-set}
    \Lambda_r :=  B(\lambda^\star,r) \cap \Lambda \ ,
\end{equation}
{\color{blue} where $B(\lambda^\star, r)$ is the closed ball of radius $r$ centered at $\lambda^\star$.}

Given a candidate primal solution $x^t$, for analysis sake only, we define the following perturbed component function value
\begin{align}\label{def:composite-linearization-center-g-tilde}
\begin{split}
    &\ \hat{g} \in \argmin_{w \in \dom{h}} h(w)+\inner{-\lambda^\star}{w}+r\|w-g(x^t)\| ,
\end{split}
\end{align} 
which exists as the objective has compact level sets. {\color{blue} In particular,   $h(w) - \inner{\lambda^\star}{w} \geq - h^*(\lambda^\star)$, by the Fenchel-Young inequality. It then holds that for any $z \in \R$, $$\{w \in \R^m : h(w) + \inner{-\lambda^\star}{w}+r\|w-g(x^t)\| \leq z\} \subseteq \{w \in \R^m : -h^*(\lambda^\star)+r\|w-g(x^t)\| \leq z\} $$
        where the larger set is bounded.} From this, for analysis sake only, we define the following associated perturbed dual variables as
$$ {\color{blue} \hat{\lambda} := \begin{cases}\lambda^\star + r\frac{g(x^t)-\hat{g}}{\|g(x^t)-\hat{g}\|} & \hat{g} \ne g(x^t)\\
\lambda^\star + r\zeta & \text{otherwise.}\end{cases}}, $$
{\color{blue} where $\zeta \in B(0,1)$ is an appropriate perturbation such that $\lambda^\star + r\zeta \in \partial h(\hat g)$, which is guaranteed by first-order optimality conditions.} 
Note that $\hat{\lambda} \in \partial h(\hat{g})$, implying $\hat{\lambda} \in \Lambda$ (since $\hat{g} \in \partial h^*({\hat{\lambda}})$), so $\hat{\lambda} \in \Lambda_r$.
The following lemma relates $(\eps,r)$-optimality to the evaluation of $Q$ at $z^t$ with respect to $(x^\star,\hat \lambda,\nabla g(x^t))$.

\begin{lemma}\label{lemma:Q-small->eps-opt} For any $z^t = (x^t ; \lambda^t,  \nu^t) \in \Z$ and $\eps > 0$, if $Q(z^t,(x^\star, \hat{\lambda}, \nabla g(x^t))) \leq \eps ,$ then $x^t$ is $(\eps,r)$-optimal~\eqref{def:optimality-conditions}.

\end{lemma}
\begin{proof}
   Let $\hat{\nu}=\nabla g(x^t)$.
   Since $\hat \lambda \in \partial h(\hat{g}) \cap \Lambda_r$, the first condition for the $(\eps,r)$-optimality of $x^t$ holds as
{\color{black}
    \begin{align*} \left[ h(\hat{g})+\inner{\hat\lambda}{g(x^t)-\hat{g}}+u(x^t)\right]-F(x^\star) &\leq \left[  h(\hat{g})+\inner{\hat\lambda}{g(x^t)-\hat{g}}+u(x^t)\right] -\hLcal(x^\star; \lambda^t, \nu^t)\\
    &= \left[   \inner{\hat\lambda}{g(x^t)}-h^*(\hat{\lambda}) +u(x^t)\right] -\hLcal(x^\star; \lambda^t, \nu^t)\\
    &= \left[   \inner{\hat\lambda}{\hat{\nu}x^t - g^*(\hat{\nu})}-h^*(\hat{\lambda}) +u(x^t)\right]-\hLcal(x^\star; \lambda^t, \nu^t)\\
    &= \hQ(z^t, (x^\star; \hat\lambda,  \hat{\nu})) \leq \eps\ , 
    \end{align*}
    where the first inequality simply bounds $F(x^\star)$ below by $\Lcal(x^\star; \lambda^t,\nu^t)$ and the following two equalities apply the Fenchel-Young inequality, holding with equality since $\hat{\lambda}\in\partial h(\hat{g})$ and $\hat{\nu} = \nabla g(x^t)$.
}
   
   To show the second condition for $(\eps,r)$-optimality holds, {\color{blue} if $\hat{g}=g(x^t)$ then this result is trivial. Otherwise,} we note that since $(x^\star,\lambda^\star)$ is a saddle point to \labelcref{eq:simple-lagrangian}, \begin{equation}\label{eq:saddle-consequence}
     0 \leq \inner{\lambda^\star}{g(x^t)}-h^*(\lambda^\star)+u(x^t) - \left[\inner{\lambda^t}{g(x^\star)}-h^*(\lambda^t)+u(x^\star)\right] \ .
 \end{equation} Consequently, 
\begin{align*}
    r\|g(x^t)-\hat{g}\|  &\leq \inner{\frac{r(g(x^t)-\hat{g})}{\|g(x^t)-\hat{g}\|}}{g(x^t)-\hat{g}}+\inner{\lambda^\star}{g(x^t)}-h^*(\lambda^\star)+u(x^t)\\&\hspace{1cm}- \left[\inner{\lambda^t}{g(x^\star)}-h^*(\lambda^t)+u(x^\star)\right]\\
    &= \inner{\hat{\lambda}}{g(x^t)}-\inner{\frac{r(g(x^t)-\hat{g})}{\|g(x^t)-\hat{g}\|}}{\hat{g}}-\inner{\lambda^\star}{\hat{g}}+\inner{\lambda^\star}{\hat{g}}-h^*(\lambda^\star)+u(x^t)\\ & \hspace{1cm}-\left[\inner{\lambda^t}{g(x^\star)}-h^*(\lambda^t)+u(x^\star)\right]\\
    &\leq \inner{\hat{\lambda}}{g(x^t)} -\inner{\hat{\lambda}}{\hat{g}}+h(\hat{g}) +u(x^t)-\left[\inner{\lambda^t}{g(x^\star)}-h^*(\lambda^t)+u(x^\star)\right]\\
      &\leq\hQ(z^t,(x^\star,\hat{\lambda},\hat{\nu})) \leq \eps  
\end{align*}
where the first inequality follows from \labelcref{eq:saddle-consequence}, and the second and third apply Fenchel-Young.  
\end{proof}

\subsection{An Approximate Dualized Aggregate Smoothness Constant} \label{subsec:Lada}

If one knew the optimal dual multipliers $\lambda^\star$, the convex composite optimization problem~\eqref{eq:WLOG-main-problem} could be rewritten as the simpler minimization problem of
\begin{equation}\label{eq:dualized-ideal-problem}
    \min_{x\in \X} \sum_{j=1}^m \lambda_j^\star g_j(x) + u(x) \ ,
\end{equation}
which can be addressed by accelerated (regularized) smooth optimization methods like FISTA~\citep{FISTA}. In this simplified problem, $\sum \lambda_j^\star g_j(x)$ is $\sum \lambda_j^\star L_j$-smooth, aggregating the individual smoothness constants weighted by the optimal dual multiplier. Without knowing $\lambda^\star$, we aim to approximate this aggregate dualized constant. Our theory instead depends on the slightly larger constant given by considering all $\lambda$ in the neighborhood of $\lambda^\star$ given by $\Lambda_r$. Given each $g_j$ is $L_j$-smooth, we denote this ``Approximate Dualized Aggregate'' smoothness constant by
\begin{equation} \label{eq:Lada-definition}
    \Lada :=  \sum_{j=1}^m (\lambda_j^\star+r) L_j \ .
\end{equation}
As $r$ tends to zero, $\Lada$ converges to the idealized value $\sum \lambda_j^\star L_j$. 
Note this only depends on the target accuracy $r>0$, not $\epsilon>0$. We include this dependence in our notation as the appropriate generalization to H\"older smooth settings given in equation \labelcref{eq:Lada-definition-general} will depend on both. Further generality will be introduced when the components possess uniform convexity in equation \labelcref{eq:Lada-definition-fully-general}.
The special case of constrained optimization minimizing $g_0(x)+u(x)$ subject to $g_j(x)\leq 0$ provides a particularly nice application to understand $\Lada$. There, $h(z_0,\dots,z_m) = z_0 + \iota_{z\leq 0}(z_1,\dots, z_m)$, so $\lambda_0^\star=1$ while $\lambda_1^\star, \dots, \lambda_m^\star$ are the optimal dual multipliers for each constraint. With $\Lada = L_0 + \sum_{j=1}^m (\lambda^\star_j+r) L_j$, the smoothness of the objective always plays a role while only the smoothness of  active constraints at the minimizer can nontrivially affect the convergence rate (that is, complementary slackness ensures that $\lambda_j^\star=0$ for each inactive constraint).

 \subsection{The Universal Fast Composite Gradient Method} \label{subsec:UFCM}
 UFCM works primarily by splitting and optimizing $\hQ(z^t, z)$ on its primal, dual, and conjugate variables separately. This process formalized below is directly analogous to the algorithm design of the ACGD-S method for smooth constrained optimization of~\cite{zhang2022solving}, {\color{blue} extended to allow a general proximal step on $h^*$ and the usage of our new $\Lada$ constant. With these established, parameter choices only need slight modifications. Hence, UFCM generalizes ACGD-S to heterogeneous and composite settings.}  We define these three components such that $\hQ(z^t, z)=\hQ_\nu(z^t, z)+\hQ_x(z^t, z)+\hQ_\lambda(z^t, z)$ as 
\begin{align*}
    \hQ_\nu(z^t, z)&=\hLcal(x^t; \lambda,  \nu)-\hLcal(x^t; \lambda,  \nu^t)=\inner{\lambda}{\nu x^t-g^*(\nu)}\boxed{-\inner{\lambda}{\nu^tx^t-g^*(\nu^t)}} \ , \\
   \hQ_x(z^t, z)&=\hLcal(x^t; \lambda^t,  \nu^t)-\hLcal(x; \lambda^t, \nu^t)=\boxed{\inner{\sum_{j=1}^m \lambda_j^t\nu_j^t}{x^t}+u(x^t)}-\inner{\sum_{j=1}^m \lambda_j^t\nu_j^t}{x}-u(x) \ , \\
   \hQ_\lambda(z^t, z)&=\hLcal(x^t; \lambda,  \nu^t)-\hLcal(x^t; \lambda^t,  \nu^t)=\inner{\lambda}{\nu^tx^t-g^*(\nu^t)}-h^*(\lambda)\boxed{-\left[\inner{\lambda^t}{\nu^tx^t-g^*(\nu^t)}-h^*(\lambda^t)\right]} \ .
\end{align*}
Each boxed term above corresponds to the component depending on the next iterate $\nu^t,x^t,\lambda^t$.
We aim to minimize each subproblem with respect to $z^t$; thus, we minimize each boxed value. Informally, UFCM proceeds by first computing a momentum step in $x$, denoted by $\tilde x^t= x^{t-1}+\theta_t(x^{t-1}-x^{t-2})$, and then computing (potentially many) proximal operator-type steps in each of $\nu, x, \lambda$ corresponding to
    {\color{blue} \begin{align*}
         \nu_j^t &\leftarrow \displaystyle\argmax_{\nu_j \in V_j} \inner{\nu_j}{\tilde{x}^t}-g_j^*(\nu_j)-\tau_tU_{g_j^*}(\nu_j;\nu_j^{t-1}), \\
            (x^t, \lambda^t) &\leftarrow \displaystyle\argmin_{x \in \X}\max_{\lambda \in \Lambda} \inner{\lambda}{\nu^tx-g^*(\nu^t)}+u(x) -h^*(\lambda) + \tfrac{\eta_t}{2}\|x-x^{t-1}\|^2
    \end{align*}}

{\color{blue} In the above, $\theta_t$ parametrizes the momentum step and the nonnegative parameters $\tau_t$ and $\eta_t$ are stepsizes for the proximal steps. Recall $U_{g_j^*}$ is the Bregman divergence generated by $g_j^*$. Note that solving $\nu^t$ utilizes a Bregman divergence instead of the standard Euclidean distance, as it can be shown recursively that this is identical to a gradient evaluation of $g_j$ at a particular averaged point~\citep[Lemma 2]{zhang2022optimal}.

Solving the second subproblem is not as simple. We utilize the sliding technique to take alternating proximal steps with respect to $x$ and $\lambda$, without addition to the gradient oracle complexity. We further employ two more nonnegative parameters, $\beta^{(t)}$ and $\gamma^{(t)}$, respectively handling the proximal steps on $u$ and $h^*$ for the inner loop iterates.} 

Formally, UFCM defined in Algorithm~\ref{alg:FCM} proceeds by iteratively applying a momentum update {\color{blue} and the above update to $\nu$} in the outer loop. The inner loop using the sliding technique to {\color{blue} apply several proximal steps to compute the above update to $(x, \lambda)$} without requiring any additional first-order evaluations of $g$. As computational notes, Line 11 saves previous iterates $y_{0}^{(t+1)}$, $\lambda_0^{(t+1)},$ and $\lambda_{-1}^{(t+1)}$ for use in the next inner loop. The subtle change from $\nu^t$ to $\nu^{t-1}$ in the two cases defined in Line 7 is common for sequential dual type algorithms using the sliding technique~\citep{zhang2022solving, lan2023optimalmethodsconvexrisk, zhang2020efficientalgorithmsdistributionallyrobust}.

\begin{algorithm}[t]
\textbf{Input}  $z^0 \in \X \times \Lambda$, outer loop iteration count $T$, and smoothness constant $\Lada$\\
 \textbf{Initialize} $x^{-1}=\underline{x}^0=y_0^{(1)}=x^0 \in \X$, $\lambda_{-1}^{(1)}=\lambda_0^{(1)}=\lambda^0 \in \Lambda$, and parameters $\{\theta_t\}, \{\eta_t\}, \{\tau_t\}$, $\{\omega_t\}$ as a function of $\Lada$
\begin{algorithmic}[1] 
\State Set $\nu^0=\nabla g(x^0)$.
\For{$t = 1,\ 2,\ 3,\ ...,\ T$}
\State Set $\underline{x}^t \gets (\tau_t \underline{x}^{t-1}+\Tilde{x}^t)/(1+\tau_t)$ where $\Tilde{x}^t = x^{t-1}+\theta_t(x^{t-1}-x^{t-2})$
\State Set $\nu^t \gets \nabla g(\underline{x}^t)$
\State Calculate inner loop iteration limit $S_t$, parameters $\beta^{(t)}$, $\gamma^{(t)}$, and $\rho^{(t)}$
\For{$s = 1, 2, ..., S_t$}
\State Set $\Tilde{h}^{(t),s}=\begin{cases}
    (\nu^t)^\top\lambda_0^{(t)}+\rho^{(t)}(\nu^{t-1})^\top(\lambda_0^{(t)}-\lambda_{-1}^{(t)}) & \text{if } s=1,\\
       (\nu^t)^\top\lambda_{s-1}^{(t)}+(\nu^{t})^\top(\lambda_{s-1}^{(t)}-\lambda_{s-2}^{(t)}) & \text{otherwise}\\
\end{cases}$
\State Solve $y_s^{(t)} \gets \displaystyle\argmin_{y \in \X}  \inner{\Tilde{h}^{(t),s}}{y}+u(y)+\frac{\eta_t}{2}\|y-x^{t-1}\|^2+\frac{\beta^{(t)}}{2}\|y-y_{s-1}^{(t)}\|^2$
\State Solve $\lambda_s^{(t)} \gets \displaystyle\argmax_{\lambda \in \Lambda}  \inner{\lambda}{\nu^t(y_s^{(t)}-\underline{x}^t)+g(\underline{x}^t)}-h^*(\lambda)-\frac{\gamma^{(t)}}{2}\|\lambda-\lambda_{s-1}^{(t)}\|^2$
\EndFor
\State Set $\lambda_0^{(t+1)} = \lambda_{S_t}^{(t)}$, $\lambda_{-1}^{(t+1)}=\lambda_{S_t-1}^{(t)}$, $y_0^{(t+1)}=y_{S_t}^{(t)}$
\State Set $x^t = \sum_{s=1}^{S_t} y_s^{(t)}/S_t$ and $\Tilde{\lambda}^t = \sum_{s=1}^{S_t} \lambda_s^{(t)}/S_t$
\EndFor\\
\Return $(\bar{x}^T,\bar{\lambda}^T) := \sum_{t=1}^T\omega_t \left( x^t,\Tilde{\lambda}^t\right)/\left(\sum_{t=1}^T \omega_t \right)$
\end{algorithmic}
\caption{Universal Fast Composite Method (UFCM)}\label{alg:FCM}
\end{algorithm}

\subsection{Guarantees for Composite Optimization with Smooth Components}\label{subsec:smooth-analysis}
 We begin this section by introducing the two oracle complexities we bound with respect to finding an $(\eps,r)$-optimal solution. We denote the gradient complexity of UFCM by $\Neps$ if for any $T \geq \Neps$, $\bar x^T$ is guaranteed to be an $(\eps,r)$-optimal point. Likewise, we denote the proximal complexity of UFCM by $\Peps$ if at most $\lceil \Peps\rceil$ proximal evaluations of $u$ and $h$ are guaranteed to be performed in the first $\lceil \Neps\rceil$ outer loop iterations of UFCM.

To ensure that UFCM converges to an $(\epsilon,r)$-optimal solution, we place several requirements on the selection of its parameters. For each outer loop $t\geq 1$, we require that
\begin{align}
     & \omega_t\eta_t \leq \omega_{t-1}\eta_{t-1}\label{cond:a1}\\
     & \omega_t\tau_t \leq \omega_{t-1}(\tau_{t-1}+1)\label{cond:a2}\\
     & \eta_{t-1}\tau_t \geq \theta_t L_{\eps,r}^{\mathtt{ADA}} \ \text{ with } \ \theta_t = \omega_{t-1}/\omega_{t}\label{cond:a3}\\
     & \eta_T(\tau_T+1) \geq L_{\eps,r}^{\mathtt{ADA}} \label{cond:a4}\\
    &\gamma^{(t)}\beta^{(t)} \geq \|\nu^t\|^2\label{cond:b3}\\
    & \Tilde{\omega}^{(t)}\beta^{(t)} \geq   \Tilde{\omega}^{(t+1)}\beta^{(t+1)}\label{cond:c1}\\
    & \Tilde{\omega}^{(t)}\gamma^{(t)} \geq   \Tilde{\omega}^{(t+1)}\gamma^{(t+1)}\label{cond:c2}\\
    &{\color{blue} \gamma^{(t)}\beta^{(t)} \geq (\rho^{(t)})^2\|\nu^{t-1}\|^2} \ \text{ with } \ \rho^{(t+1)}={\Tilde{\omega}^{(t)}}/{\Tilde{\omega}^{(t+1)}}\label{cond:c3}
\end{align}
\vspace{0.4cm}
where $\Tilde{\omega}^{(t)} := \omega_t/S_t$ denotes the aggregate weights.

Although our algorithm converges for any selection satisfying these requirements, optimized performance follows from particular choices. In particular, our main convergence guarantee below requires knowledge of an upper bound on $\Lada$ to set parameters. Some of our convergence guarantee corollaries additionally assume knowledge of positive bounds on the initial distances to a saddle point $D_x \geq \|x^0-x^\star\|$ and $D_\lambda \geq \|\lambda^0-\lambda^\star\|$.

As a first result, we establish that a careful setting of stepsizes ensures that the primal iterates are always bounded and that the dual iterates are bounded if $h$ is $L_h$-smooth\footnote{We will abuse notation in the setting of general, nonsmooth $h$, saying $h$ is $L_h=\infty$-smooth in this limiting case.}. The parameters of \Cref{alg:FCM} below are further parameterized by the choice of two balancing parameters $C$ and $\Delta$.
\begin{proposition}\label{lemma:h-smooth-lambda-iterate-bound}
    Consider any problem of the form~\labelcref{eq:WLOG-main-problem} and constants $\Delta, C, \epsilon, r > 0$, and suppose \Cref{alg:FCM} is run with outer loop stepsizes set as
    \begin{equation}\label{eq:outerloop-stepsizes}
        \tau_t = \frac{t-1}{2}, \hspace{0.2cm} \eta_t = \frac{L_{\eps,r}^{\mathtt{ADA}}}{\tau_{t+1}}, \hspace{0.2cm} \theta_t = \frac{\tau_t}{\tau_{t-1}+1}, \hspace{0.2cm}  \omega_t = t \ ,
    \end{equation}
    and inner loop stepsizes set as  
    \begin{equation}\label{eq:innerloop-stepsizes}
        \rho^{(t)}=
   {\Tilde{M}_t}/{\Tilde{M}_{t-1}}, \hspace{0.3cm} \beta^{(t)} = C\Tilde{M}_t, \hspace{0.3cm} \gamma^{(t)} = \frac{\tilde{M}_t^2}{\beta^{(t)}}=\frac{\Tilde{M}_t}{C} \ ,
    \end{equation}
    with $M_t = \|\nu^t\|$, $S_t=\lceil M_t \Delta t \rceil, \hspace{0.1cm} \tilde{M}_t = \frac{S_t}{\Delta t}$. Then 
    \begin{equation}\label{eq:primal-iterate-bound}
        \|x^t-x^\star\|^2 \leq \frac{1}{2L_{\eps,r}^{\mathtt{ADA}}}\left[(C/\Delta+2L_{\eps,r}^{\mathtt{ADA}})\|x^0-x^\star\|^2+\frac{1}{C \Delta}\|\lambda^0-\lambda^\star\|^2\right] \ .
    \end{equation} 
    Furthermore, 
    if $h$ is $L_h$-smooth, then for averaged iterates $\Tilde \lambda^{t}$ computed each loop, \begin{equation}\label{eq:dual-iterate-bound}
        \|\Tilde\lambda^t-\lambda^\star\|^2 \leq  L_h (M\Delta+1)\left[(C/\Delta + 2 L_{\eps,r}^{\mathtt{ADA}})\|x^0-x^\star\|^2+\frac{1}{C\Delta} \|\lambda^0-\lambda^\star\|^2\right] \ .
    \end{equation}
    where $M$ is an upper bound for $\|\nabla g(x)\|$ in the neighborhood outlined above \labelcref{eq:primal-iterate-bound}.
\end{proposition}

Moreover, under such choices, the following theorem explicitly bounds the number of gradient and proximal oracle calls required to reach any target $(\epsilon,r)$-optimality.

\begin{theorem}
\label{thm:smooth-composite-convergence}
    Consider any problem of the form~\labelcref{eq:WLOG-main-problem} with each $g_j$ being $L_j$-smooth,  and constants $\Delta, C, \epsilon, r > 0$. Then \Cref{alg:FCM} with stepsizes \labelcref{eq:outerloop-stepsizes} and \labelcref{eq:innerloop-stepsizes} must find an $(\eps,r)$-optimal solution~\labelcref{def:optimality-conditions} with complexity bounds   
    \begin{equation}\label{eq:Smooth-UFCM-thm-bounds}
        \Neps = \sqrt{\frac{(C/\Delta+2L_{\eps,r}^{\mathtt{ADA}})D_x^2+{\color{blue}2}/(C\Delta)(D_\lambda^2+r^2)}{\eps}}, \quad \Peps = \lceil\Neps \rceil+\lceil\Neps \rceil^2\Delta M ,
    \end{equation} 
     {\color{blue} where $M$ is an upper bound for $\|\nabla g(x)\|$ in the neighborhood outlined in \labelcref{eq:primal-iterate-bound}}  
\end{theorem}

The following corollaries simplify the above bounds by considering particular choices of $\Delta$, $C$, and $r$. The first corollary presents an upper bound in terms of a primal-dual distance while avoiding reliance on knowledge of any upper bounds on initial distances to optimality. The second corollary provides an improved gradient complexity bound depending only on primal distances at the cost of requiring knowledge of upper bounds on the initial primal and dual distances to a saddle point. Our extended theory in \cref{section:Holder-Compositions} and \cref{section:UC-restarting-comp} will focus on generalizing this second, stronger result. The remainder of this section is dedicated to proving these results.
\begin{corollary}\label{cor:smooth-composite-rates}
      For any $\epsilon>0$, setting $C=\sqrt{2}/2$, $\Delta = \frac{\sqrt{2}}{4L_{\eps,r}^{\mathtt{ADA}}}$ and $r=\sqrt{\eps}$, Algorithm~\ref{alg:FCM} with stepsizes \labelcref{eq:outerloop-stepsizes} and \labelcref{eq:innerloop-stepsizes} must find an $\eps$-optimal solution with complexity bounds
      $$\Neps = {\bigO}\left(\sqrt{\frac{ L_{\eps,r}^{\mathtt{ADA}}(D_x^2+D_\lambda^2+\eps)}{\eps}}\right), \quad \Peps = {\bigO}\left(\sqrt{\frac{ L_{\eps,r}^{\mathtt{ADA}}(D_x^2+D_\lambda^2+\eps)}{\eps}}+\frac{ M(D_x^2+D_\lambda^2+\eps)}{\eps}\right)$$
    where $M$ is an upper bound on $\|\nabla g(x)\|$ for all $x \in B(x^\star,\sqrt{2(D_x^2+D_\lambda^2)})$  
\end{corollary}

\begin{corollary}\label{cor:smooth-composite-rates-primal-dist}
    For any $0 < \epsilon \leq \min\{1,12\Lada D_x^2\}$, setting $C=D_\lambda/D_x$, $\Delta = C/2L_{\eps,r}^{\mathtt{ADA}}$, and $r=D_\lambda \sqrt{\eps}$, Algorithm~\ref{alg:FCM} with stepsizes \labelcref{eq:outerloop-stepsizes} and \labelcref{eq:innerloop-stepsizes} must find an $\eps$-optimal solution with complexity bounds $$\Neps = {\bigO}\left(\sqrt{\frac{ L_{\eps,r}^{\mathtt{ADA}}D_x^2}{\eps}}\right), \quad \Peps = {\bigO}\left(\sqrt{\frac{ L_{\eps,r}^{\mathtt{ADA}}D_x^2}{\eps}}+\frac{ MD_xD_\lambda}{\eps}\right) \ ,$$  
   where $M$ is an upper bound on $\|\nabla g(x)\|$ for all $x \in B(x^\star,\sqrt{3D_x^2}).$ 
\end{corollary}

\subsection{Analysis of UFCM for Compositions with Smooth Components}
Our theory primarily follows from a sequence of three lemmas, which directly extend equivalent results developed for the case of smooth constrained optimization by~\citep{zhang2022solving}. {\color{blue} In our analysis, note that we select $\hat g$, our analytical proxy for $g(x^t)$, differently than the projection choice used by Zhang and Lan.} Throughout, we let $L_h\in (0,\infty]$ denote the smoothness constant of $h$, set to be $\infty$ if $h$ is nonsmooth, as occurs in the special case of constrained optimization. For each result, we refer to the paralleled proof in their special case. For results requiring generalization, we defer the proofs to Appendix~\ref{section:appendix-smooths-comp}. Our results show that the analysis technique of Zhang and Lan is quite robust, generalizing to compositions, managing new $h^*$ terms, and benefiting from any smoothness in $h$.

The first lemma provides a useful smoothness bound on the Lagrangian from our Approximate Dualized Aggregate smoothness constant. 
\begin{lemma}[Lemma 2,~\citet{zhang2022solving}]\label{lemma:zhang-smoothness-growth-lemma}
    If each $g_j$ is $L_j$-smooth, then
    \begin{align*}
        \inner{\lambda}{U_{g^*}(\nu;\hat{\nu})} \geq \frac{1}{2\Lada}\left\|\sum_{j=1}^m \lambda_j(\nu_j-\hat{\nu}_j)\right\|^2, \ \forall \lambda \in \Lambda_r, \ \forall \nu, \hat \nu \in \{\nabla g(x) : x \in \X\} \ .
    \end{align*}   
\end{lemma}
Next, we provide a general convergence bound on the $Q_x$ and $Q_\lambda$ functions associated with the primal and dual variables, extending the result of~\citep[Equation (4.19)]{zhang2022solving} and proven in Appendix~\ref{section:appendix-smooths-comp}. When $h$ is nonsmooth (i.e., $L_h = +\infty$), the quantity $1/L_h$ below should be interpreted at zero.
\begin{lemma}\label{lemma:x-lambda-convergence}
        Suppose the stepsizes satisfy~\eqref{cond:a1}-\eqref{cond:c3}, and let $z^t := (x^t; \Tilde{\lambda}^t, \nu^t)$ denote the iterates of \Cref{alg:FCM}. Then $z^t$ satisfies the following for any $z=(x;\lambda,\nu) \in \X \times \Lambda_r \times V$, \begin{align*}\begin{split}
    \sum_{t=1}^T \omega_t [\hQ_x(z^t,z)&+\hQ_\lambda(z^t,z)]+\sum_{t=1}^T\sum_{s=1}^{S_t}\frac{\omega_t}{S_t}\muhstar\|\lambda_s^{(t)}-\lambda\|^2+\sum_{t=1}^T \frac{\omega_t\eta_t}{2} \|x^t-x^{t-1}\|^2\\&+\frac{\omega_T\eta_T}{2}\|x^T-x\|^2-\frac{\omega_1\eta_1}{2}\|x^0-x\|^2 \leq \frac{\Tilde{\omega}^{(1)}}{2}\left(\gamma^{(1)}\|\lambda_0^{(1)}-\lambda\|^2+\beta^{(1)}\|y_0^{(1)}-x\|^2\right) \ .
    \end{split}
\end{align*}
\end{lemma}

Lastly, we provide a general convergence bound on the $Q_\nu$ function associated with the conjugate variables $\nu$, which requires only mild modifications from the analysis of Zhang and Lan~\citep[Proposition 2]{zhang2022solving}, proven in Appendix~\ref{section:appendix-smooths-comp}. 
\begin{lemma}\label{lemma:nu-convergence-bound} Suppose the stepsizes satisfy~\eqref{cond:a1}-\eqref{cond:c3}. Then $z^t$ satisfies the following for any $z=(x;\lambda,\nu) \in \X \times \Lambda_r \times V$,
     \begin{align*}\begin{split}
              \sum_{t=1}^T \omega_t\left[\hQ_\nu(z^t,z)\right]
            &\leq - \left[\omega_T(\tau_T+1)\left(\sum_{j=1}^m \lambda_jU_{g_j^*}(\nu_j;\nu_j^T)\right)-\omega_T\inner{\sum_{j=1}^m \lambda_j(\nu_j-\nu_j^T)}{x^T-x^{T-1}}\right]\\
            &-\sum_{t=2}^T\left[\omega_t\tau_t\left(\sum_{j=1}^m \lambda_jU_{g_j^*}(\nu_j^t;\nu_j^{t-1})\right)-\omega_{t-1}\inner{\sum_{j=1}^m \lambda_j(\nu_j^{t-1}-\nu_j^{t})}{(x^{t-1}-x^{t-2})}\right]\\
            &+\omega_1\tau_1\inner{\lambda}{U_{g^*}(\nu,\nu^0)} \ .
            \end{split}
        \end{align*}
\end{lemma}

Combining these three lemmas gives a single convergence result for the entire gap function $Q$. This result looks nearly identical in form to Proposition 2 from~\citep{zhang2022solving} and is proven in Appendix~\ref{section:appendix-smooths-comp}. From this proposition, we then prove our claimed compactness and convergence guarantees in Proposition~\ref{lemma:h-smooth-lambda-iterate-bound} and \cref{thm:smooth-composite-convergence}. 
\begin{proposition}\label{prop:Q-convergence-bound} Consider any problem of the form \labelcref{eq:WLOG-main-problem} with stepsizes satisfying \labelcref{cond:a1}-\labelcref{cond:c3}. Then for any $z = \displaystyle(x; \lambda, \nu) \in \X \times \Lambda_r \times V$,
\begin{align*}
    &\sum_{t=1}^T \omega_t\hQ(z^t,z)+\sum_{t=1}^T\sum_{s=1}^{S_t}\frac{\omega_t}{S_t}\muhstar\|\lambda_s^{(t)}-\lambda\|^2+\frac{\omega_T\eta_T}{2}\|x^T-x\|^2\\ &\leq \frac{\Tilde{\omega}^{(1)}\beta^{(1)}+\omega_1\eta_1}{2}\|x^0-x\|^2+\frac{\Tilde{\omega}^{(1)}\gamma^{(1)}}{2}\|\lambda^0-\lambda\|^2+\omega_1\tau_1\inner{\lambda}{U_{g^*}(\nu,\nu^0)} \ .
\end{align*}
\end{proposition}

\paragraph{Proof of Proposition~\ref{lemma:h-smooth-lambda-iterate-bound}.}
  First, we claim the proposed stepsizes in~\eqref{eq:outerloop-stepsizes} and~\eqref{eq:innerloop-stepsizes} satisfy the necessary conditions \labelcref{cond:a1}-\labelcref{cond:c3} for the preceding proposition and lemmas to apply. Each of these conditions can be directly checked: See \citep[Theorem 5]{zhang2022solving} for equivalent verifications in the simplified setting of constrained optimization, only differing in that we consider a generic $C$ rather than fixing $C = \frac{\|\lambda^\star\|+r}{\|x^0-x^\star\|}$ in our choice of $\beta^{(t)}=C\Tilde{M}_t$.

  Note that from \cref{prop:Q-convergence-bound}, the assumption that $L_h \in (0,\infty]$, and the fact that $\tau_1=0$,  \begin{equation}\label{eq:simplified-Q-bound}
      \sum_{t=1}^T \omega_t\hQ(z^t,z) + \frac{\omega_T\eta_T}{2}\|x^T-x\|^2 \leq \frac{\Tilde{\omega}^{(1)}\beta^{(1)}+\omega_1\eta_1}{2}\|x^0-x\|^2+\frac{\Tilde{\omega}^{(1)}\gamma^{(1)}}{2}\|\lambda^0-\lambda\|^2 \ .
  \end{equation}  Furthermore, since this holds for all $z \in \X \times \Lambda_r \times V$, we consider the saddle point $z^\star$. As a saddle point, $Q(z^t,z^\star) \geq 0$. Using the stepsize conditions \labelcref{eq:outerloop-stepsizes} and \labelcref{eq:innerloop-stepsizes}, \begin{equation}\label{eq:UFCM-stepsize-t=1}
         \omega_1=1, \quad \Tilde{\omega}^{(1)}=\frac{1}{S_1}, \quad \eta_1=2L_{\eps,r}^{\mathtt{ADA}}, \quad
        \beta^{(1)} = \frac{CS_1}{\Delta}, \quad \gamma^{(1)}=\frac{\Tilde{M}_1^2}{\beta^{(1)}}=\frac{S_1}{C\Delta} \ ,
    \end{equation} gives the claimed bound on $x^T$.
  Therefore, for $t \geq 1$, each $x^t$ lies in the desired bounded neighborhood around $x^\star$. Note that $\underline{x}^t \in \conv{x^0, \dots, x^{t-1}}$, so $\underline{x}^t$ is in the same neighborhood. As a result, $M_t = \|\nabla g(\underline{x}^t)\|$ is bounded uniformly by $M$. 
  
  For the dual iterates, \cref{prop:Q-convergence-bound} and the nonnegative of $Q(z^t; z^\star)$ and the norm ensures

 $$\sum_{t=1}^T\sum_{s=1}^{S_t}\frac{\omega_t}{S_t}\muhstar\|\lambda_s^{(t)}-\lambda^\star\|^2 \leq \frac{\Tilde{\omega}^{(1)}\beta^{(1)}+\omega_1\eta_1}{2}\|x^0-x^\star\|^2+\frac{\Tilde{\omega}^{(1)}\gamma^{(1)}}{2}\|\lambda^0-\lambda^\star\|^2 \ .$$
    For any $s, t \geq 1$, one can bound $$\|\lambda_{s}^{(t)}-\lambda^\star\|^2 \leq \frac{2L_hS_t}{\omega_t }\left[\frac{\Tilde{\omega}^{(1)}\beta^{(1)}+\omega_1\eta_1}{2}\|x^0-x^\star\|^2+\frac{\Tilde{\omega}^{(1)}\gamma^{(1)}}{2}\|\lambda^0-\lambda^\star\|^2\right] \ .$$   
    Utilizing the values in \labelcref{eq:UFCM-stepsize-t=1}, bounding $M_t$ {\color{blue} above} by $M$, {\color{blue} $S_t/\omega_t$  above by $M\Delta + 1 $ for $t \geq 1$}, and noting $\Tilde \lambda^t$ lies in the convex hull of $\{\lambda_s^{(t)}\}$ yields the claimed bound on $\lambda^t$.

\paragraph{Proof of \cref{thm:smooth-composite-convergence}.} Let $\bar{z}^T=(\bar{x}^T; \bar{\lambda}^T, \bar{\nu}^T)$ where
        \begin{equation}\label{eq:averaging-scheme} \bar{x}^T = \sum_{t=1}^T \omega_t x_t/\sum_{t=1}^T \omega_t \quad
     \bar{\lambda}^T = \sum_{t=1}^T \omega_t \Tilde{\lambda}^t/\sum_{t=1}^T \omega_t,  \quad
        \bar{\nu}_j^T = \begin{cases}
            \nabla g_j(x^0) & \text{if } \Tilde{\lambda}_j^t=0 \text{ for all } t,\\
            \sum_{t=1}^T \omega_t\Tilde{\lambda}_j^t \nu_j^t/\sum_{t=1}^T \omega_t\Tilde{\lambda}_j^t & \text{otherwise.}
        \end{cases} \end{equation}
        Consequently, 
    \begin{align*}
        \sum_{t=1}^T \omega_t \hLcal(x; \Tilde{\lambda}^t, \nu^t) &= \sum_{t=1}^T \omega_t \left[u(x)+\inner{\Tilde{\lambda}^t}{ \nu^t x-g^*(\nu^t)}-h^*(\Tilde \lambda^t)\right]\\
        &= \sum_{t=1}^T \omega_t u(x)+\sum_{t=1}^T \omega_t \sum_{j=1}^m \Tilde{\lambda}_j^t(\inner{\nu_j^t}{x}-g_j^*(\nu_j^t))-\sum_{t=1}^T \omega_t h^*(\Tilde \lambda^t)\\
        &\leq \left(\sum_{t=1}^T \omega_t\right) u(x)+ \left(\sum_{t=1}^T \omega_t\right)\left[ \sum_{j=1}^m \bar{\lambda}_j^T \left(\inner{\bar{\nu}_j^T}{x}-g_j^*(\bar{\nu}_j^T)\right)- h^*(\bar \lambda^T)\right]\\&= \left(\sum_{t=1}^T \omega_t\right)\hLcal(x; \bar{\lambda}^T, \bar{\nu}^T) \ ,
    \end{align*}
where the inequality follows from Jensen's inequality. Similarly, Jensen's inequality ensures that $$\left(\sum_{t=1}^T \omega_t\right) \hLcal(\bar{x}^T; \lambda, \nu) \leq \sum_{t=1}^T \omega_t \hLcal(x^t; \lambda, \nu) \ .$$
    Therefore, we have for all $z \in \Z$
    \begin{align}\label{eq:Gap-Jensen's}
    \begin{split}
        \left(\sum_{t=1}^T \omega_t\right) \hQ(\bar{z}^T, z) &= \left(\sum_{t=1}^T \omega_t\right)\left[\hLcal(\bar{x}^T;\lambda,\nu) -\hLcal(x;\bar{\lambda}^T, \bar{\nu}^T )\right]\\ &\leq \sum_{t=1}^T \omega_t \left[ \hLcal(x^t; \lambda, \nu)-\hLcal(x; \Tilde{\lambda}^t, \nu^t)\right]=\sum_{t=1}^T \omega_t \hQ(z^t,z) \ .
        \end{split}
    \end{align}
    Utilizing the above inequality, the bound demonstrated in \labelcref{eq:simplified-Q-bound}, the nonnegativity of the norm, our distance bounds, as well as the triangle inequality,

    \begin{equation*}\label{eq:Q-convergence-unsimplified}
        \left(\sum_{t=1}^T \omega_t\right)\hQ(\bar{z}^T,(x^\star;\lambda,\nu)) \leq \frac{\Tilde{\omega}^{(1)}\beta^{(1)}+\omega_1\eta_1}{2}D_x^2+ {\Tilde{\omega}^{(1)}\gamma^{(1)}}(D_\lambda^2+r^2), \text{ for all } \lambda \in \Lambda_r, \ \nu \in V  \ .
    \end{equation*}

   Finally, we bound $\sum_{t=1}^T \omega_t$ by $T^2/2$ and substitute the values from \labelcref{eq:UFCM-stepsize-t=1} into the above expression. Considering \cref{lemma:Q-small->eps-opt}, it suffices to bound the above by $\eps$. Since each outer loop of UFCM computes only one gradient of $g_j$, $\Neps = T$ where $T > 0 $ solves $$\frac{(C/\Delta+2\Lada) D_x^2 + {\color{blue}2}/(C\Delta)(D_\lambda^2+r^2)}{T^2} = \eps \ ,$$resulting in the complexity bound for $\Neps$.
  Noting that each inner loop performs only one proximal step on $u$ and $h^*$,  \begin{align}\label{eq:general-P_eps}
        \Peps &\leq \sum_{t=1}^{\lceil\Neps \rceil} S_t= \sum_{t=1}^{\lceil\Neps \rceil} \left\lceil M_t \Delta t \right\rceil\leq \sum_{t=1}^{\lceil\Neps \rceil} (1+M\Delta t)\leq (\lceil\Neps \rceil + {\lceil\Neps \rceil^2} M \Delta) \ .
    \end{align}

   \section{Compositions with Heterogeneously H\"older Smooth Components}\label{section:Holder-Compositions}
Our convergence theory for problems with smooth components developed so far extends to instances where each $g_j$ has H\"older continuous gradient with individual exponents. Recall, we say a function $f$ is $(L,p)$-\Holder smooth with $L \geq 0$ and $p \in [0,1]$ if the function satisfies \begin{equation*}
    \|\nabla f(x) -\nabla f(y)\| \leq L\|x-y\|^p, \ \forall x,y \in \dom{f} \ . 
\end{equation*}
When $p=1$, we recover standard $L$-smoothness, and when $p=0$, the function $f$ is Lipschitz.  Therefore, \Holder smoothness lets one interpolate between smooth and nonsmooth functions.

\subsection{A Universal Definition of the Approximate Dualized Aggregate Smoothness}

The key result facilitating the design of methods universally applicable to H\"older smooth problems is proven by~\citep[Lemma 1]{Nesterov_2014}, previously introduced here as Lemma~\ref{lemma:Nesterov-universal-quad}. We further utilize the subsequent cocoercive extension, introduced here as \cref{lemma:cocoercive-Holder}, which was proven by~\citep[Lemma 1]{li2024simpleuniformlyoptimalmethod}. These results show, for any fixed tolerance $\delta > 0$, there exists a constant $L_\delta = \left[\frac{1-p}{1+p}\frac{1}{\delta}\right]^{\frac{1-p}{1+p}}L^{\frac{2}{1+p}}$ such that the standard quadratic upper bound inequality or cocoercivity inequality of smooth convex functions hold (up to $\delta$) for any $(L,p)$-\Holder smooth function.
Supposing each $g_j$ is $(L_j,p_j)$-\Holder smooth, define a general smoothness constant for fixed tolerance $\delta > 0$ as \begin{equation}\label{eq:general-heterogeneous-smoothness-constant}
    L_{\delta,r} := \sum_{j=1}^m \left[\left[\frac{1-p_j}{1+p_j}\cdot\frac{m}{\delta}\right]^\frac{1-p_j}{1+p_j}\left[(\lambda_j^\star+r) \cdot L_j\right]^\frac{2}{1+p_j}\right] \ .
\end{equation}
Noting each $\lambda_j g_j$ is $(\lambda_jL_j,p_j)$-H\"older smooth, this constant is large enough to ensure that
Lemmas~\ref{lemma:Nesterov-universal-quad} and~\ref{lemma:cocoercive-Holder} apply to each component with tolerance $\delta/m$. Summing these $m$ components, the ideal dualized problem~\eqref{eq:dualized-ideal-problem} arising if one knew the optimal dual multipliers is approximated within tolerance $\delta$. We utilize this general constant to motivate our unifying theory.

Our universal definition for the Approximate Dualized Aggregate smoothness constant $\Lada$, generalizing the smooth case previously defined in~\eqref{eq:Lada-definition}, then follows from careful selection of this $\delta$ tolerance. To achieve optimal convergence guarantees, we require the following implicit choice for the definition of $\delta$. Given an initialization $D_x \geq \|x^0-x^\star\|$ and choices of $\epsilon,r >0 $, we define the Approximate Dualized Aggregate smoothness constant as the unique positive root to the following equation
\begin{equation} \label{eq:Lada-definition-general}
   \Lada := \left\{L^\mathtt{ADA} > 0 : L^\mathtt{ADA} = \sum_{j=1}^m \left[\frac{1-p_j}{1+p_j}\cdot \frac{m\sqrt{L^\mathtt{ADA}}}{\eps}\cdot \frac{2\sqrt{6}D_x}{\sqrt{\eps}}\right]^\frac{1-p_j}{1+p_j}\left[(\lambda_j^\star+r) L_j\right]^\frac{2}{1+p_j}\right\} \ .
\end{equation}
 We note that this value is precisely $L_{\delta,r}$ \labelcref{eq:general-heterogeneous-smoothness-constant} with specialized $\delta = \eps/\sqrt{24\Lada D_x^2/\eps}$. As each $p_j$ tends to one, the associated coefficient tends to one, becoming independent on $\eps$. As all $p_j $ tend to one, the above sum defining $\Lada$ tends to $ \sum_{j=1}^m (\lambda_j^\star+r)L_j$, recovering our previous definition \labelcref{eq:Lada-definition} as a special case. 

 \begin{lemma}\label{lemma:Lada-nonincreasing}
     The Approximate Dualized Aggregate smoothness constant $\Lada$ as defined in \labelcref{eq:Lada-definition-general} is nonincreasing with respect to $\eps$.
 \end{lemma}
\begin{proof}
     Consider $\eps' \geq \eps > 0$. Rearranging the definitions of $\Lada$ and $L_{\eps',r}^\mathtt{ADA}$ ensure that $$\sum_{j=1}^m C_j (\Lada)^{-\frac{1+3p_j}{2(1+p_j)}}\eps^{-\frac{3-3p_j}{2+2p_j}} = 1, \quad \sum_{j=1}^m C_j (L_{\eps',r}^\mathtt{ADA})^{-\frac{1+3p_j}{2(1+p_j)}}(\eps')^{-\frac{3-3p_j}{2+2p_j}} = 1\ ,$$
  with $C_j=\left[\frac{1-p_j}{1+p_j}\cdot 2\sqrt{6}mD_x\right]^{\frac{1-p_j}{1+p_j}}\left[(\lambda_j^\star+r)L_j\right]^\frac{2}{1+p_j}$. Since each $p_j \in [0,1]$, it follows that $(\eps')^{-\frac{3-3p_j}{2+2p_j}} \leq (\eps)^{-\frac{3-3p_j}{2+2p_j}}$, and for the above sums to equal $1$, it must hold that the positive solution $L_{\eps',r}^\mathtt{ADA} \leq \Lada$.
\end{proof}
\subsection{Guarantees for Composite Optimization with Heterogeneous Components}
Importantly, note that we are not making any modifications to UFCM in this section. The algorithm does not require knowledge of the implicitly defined $\delta$ value or any $(L_j,p_j)$ pairs. They are for analysis only. Instead, the UFCM algorithm only relies on an estimate of $\Lada$, which could be guessed via a geometric parameter schedule without attempting to approximate $\delta$ or any of the $(L_j,p_j)$ pairs. {\color{blue} (See Remark~\ref{rmk:tradeoff-lada-dist}.)}

Next, we present our convergence theory, further justifying the choice of the implicit constant definition~\eqref{eq:Lada-definition-general}. Our theory provides guarantees for any choice of $\delta$ and approximate smoothness constant $L_{\delta,r}$. However, this choice $\Lada$ optimizes the strength of our guarantee over all $\delta$.
{\color{blue} We first generalize \cref{lemma:h-smooth-lambda-iterate-bound}, ensuring the iterates of UFCM stay bounded in this heterogeneously smooth setting, accounting for a slightly larger radius due to the additive error term. We defer the proof to \cref{section:appendix-heterogeneous-comp-proofs}.

\begin{proposition}\label{prop:compactness of iterates general}
      Consider any problem of the form~\labelcref{eq:WLOG-main-problem} and constants $\Delta, C, \epsilon, r > 0$, and suppose \Cref{alg:FCM} is run with outer loop stepsizes~\labelcref{eq:outerloop-stepsizes}
    and inner loop stepsizes~\labelcref{eq:innerloop-stepsizes}, then for all $t \leq \sqrt{\frac{24\Lada D_x^2}{\eps}}$,
    \begin{equation}\label{eq:primal-iterate-bound-general}
        \|x^t-x^\star\|^2 \leq \frac{1}{2L_{\eps,r}^{\mathtt{ADA}}}\left[(C/\Delta+50L_{\eps,r}^{\mathtt{ADA}})D_x^2+\frac{1}{C \Delta}D_\lambda^2\right]  \ .
    \end{equation} 
    Furthermore, 
    if $h$ is $L_h$-smooth, then for averaged iterates $\Tilde \lambda^{t}$ computed each loop, \begin{equation}\label{eq:dual-iterate-bound-general}
        \|\Tilde\lambda^t-\lambda^\star\|^2 \leq  L_h (M\Delta+1)\left[(C/\Delta + 50 L_{\eps,r}^{\mathtt{ADA}})D_x^2+\frac{1}{C\Delta} D_\lambda^2\right] \ .
    \end{equation}
    where $M$ is an upper bound for $\|\nabla g(x)\|$ in the neighborhood outlined above \labelcref{eq:primal-iterate-bound-general}.
\end{proposition}}

\begin{theorem}\label{thm:heterogeneous-composite-convergence}
    Consider any problem of the form~\labelcref{eq:WLOG-main-problem}  with each $g_j$ being $(L_j,p_j)$-H\"older smooth,  and constants $\Delta, C, \epsilon, r > 0$. Then \Cref{alg:FCM} with stepsizes \labelcref{eq:outerloop-stepsizes} and \labelcref{eq:innerloop-stepsizes} must find an $(\eps,r)$-optimal solution~\labelcref{def:optimality-conditions} with  complexity bounds 
  {\color{blue}  \begin{equation}\label{eq:Heterogeneous-complexity-bounds}
        \Neps = \sqrt{\frac{(2C/\Delta+4L_{\eps,r}^{\mathtt{ADA}})D_x^2+4/(C\Delta)(D_\lambda^2+r^2)}{\eps}}, \quad \Peps = \lceil\Neps \rceil+\lceil\Neps \rceil^2\Delta M ,
    \end{equation}} where $M$ is an upper bound for $\|\nabla g(x)\|$ in the neighborhood outlined in \labelcref{eq:primal-iterate-bound-general}.  
    
For target accuracy $0 < \eps \leq \min\{1, 24\Lada D_x^2\}$ with $r=D_\lambda \sqrt{\eps}$,  and setting $C=D_\lambda/D_x$ and $\Delta = C/(2\Lada)$,  these bounds simplify to
    \begin{equation}\label{eq:Heterogeneous-complexity-bounds-simpler}
        \Neps = \sqrt{\frac{24\Lada D_x^2}{\eps}}, \quad \Peps = \sqrt{\frac{24\Lada D_x^2}{\eps}} + \frac{48MD_xD_\lambda}{\eps} + 1 \ ,
    \end{equation} where $M$ is an upper bound on $\|\nabla g(x)\|$ for all $x \in B(x^\star,\sqrt{{\color{blue}27}D_x^2}).$  
\end{theorem}

The preceding theorem demonstrates how the complicated nature of heterogeneous optimization can be simplified to look analogous to the standard accelerated rate of unconstrained smooth optimization. Of course, in general, this rate is not  ${\bigO}(1/\sqrt{\eps})$ as $\Lada$ may be non-constant in $\eps$. The aggregating parameter $\Lada$ provides the key mechanism to provide a single unifying, universal guarantee.

Further, $\Lada$ serves as a universal tool to recover optimal {\color{blue} rates in terms of gradient oracle complexity} for known problem classes. First note that \cref{thm:heterogeneous-composite-convergence} recovers the optimal  rates from smooth compositions from \cref{section:Smooth-Composite} within a factor of two as our more general definition of $\Lada$ reduces to the previous definition~\eqref{eq:Lada-definition} when $p_j=1$. The following corollaries demonstrate $\Lada$'s ability to recover optimal results {\color{blue} in terms of gradient oracle complexity} from the literature for minimizing a single \Holder smooth function~\citep{Nemirovski_restarting}, rates for a heterogeneous sum of \Holder smooth terms~\citep{grimmer2023optimal}, and rates for smooth constrained optimizations~\citep{zhang2022solving}.

\begin{corollary}\label{cor:Lada-recovers-HS}
    Consider minimizing $g_0(x) + u(x)$ where $g_0$ is $(L,p)$-\Holder smooth with initial distance bound $D_x \geq \|x^0-x^\star\|$, initialization $\lambda^0= 1$, and target accuracy $0 < \eps \leq \min\{1,2\sqrt{6}LD_x^{1+p}\}$. Then \Cref{alg:FCM} finds an $\eps$-optimal solution with complexity bounds $$\Neps = \Peps = {\bigO}\left(\left(\frac{L}{\eps}\right)^\frac{2}{1+3p}D_x^{\frac{2+2p}{1+3p}}\right) \ .$$
\end{corollary}

\begin{proof}
    For minimizing a single function, $h(z)=z$, and $\lambda^0=\lambda^\star = 1$. Therefore, $D_\lambda$ {\color{blue} and $r$ }can be arbitrarily small, as well as $\Delta=D_\lambda/D_x$. Recall $S_t = \lceil M_t \Delta t\rceil$, so we set $S_t = 1$ for all $t$, and each inner loop of UFCM only computes a single proximal step on $u$ and $h^*$. Therefore, $\Neps = \Peps$.
    
    We now focus on the gradient oracle complexity. It is straightforward to check that our hypothesis enforces $\eps \leq 2\sqrt{6}LD_x^{1+p} \leq L\left(\frac{1-p}{1+p}\right)^\frac{1-p}{2}(2\sqrt{6}D_x)^{1+p}  \leq  24\Lada D_x^2$, which allows us to apply \cref{thm:heterogeneous-composite-convergence}. Considering our Approximate Dualized Aggregate constant yields $$\Lada = (1+r)^\frac{4}{1+3p}\left[\frac{1-p}{1+p}\cdot \frac{2\sqrt{6}D_x}{\eps\sqrt{\eps}}\right]^\frac{2-2p}{1+3p} L^\frac{4}{1+3p} \ .$$ We then conclude 
    \begin{equation*}
        \Neps = {\bigO}\left(\sqrt{\frac{\Lada D_x^2}{\eps}}\right) = {\bigO}\left(\sqrt{\left( \frac{D_x}{\eps\sqrt{\eps}}\right)^\frac{2-2p}{1+3p}\frac{ L^\frac{4}{1+3p} D_x^2}{\eps}}\right) ={\bigO}\left(\left(\frac{L}{\eps}\right)^\frac{2}{1+3p}D_x^{\frac{2+2p}{1+3p}}\right)  
    \end{equation*}
    where the first equality considers \labelcref{eq:Heterogeneous-complexity-bounds-simpler}, the second substitutes $\Lada$, and the third simplifies to recover \labelcref{eq:optimal-holderSmooth-rate}.
\end{proof}

\begin{corollary}\label{cor:Lada-general-implicit-formula}
   In the setting of \cref{thm:heterogeneous-composite-convergence}, \Cref{alg:FCM} finds an $\eps$-optimal solution with complexity bounds $\Neps $ and $\Peps = \lceil\Neps\rceil+\lceil\Neps\rceil^2M\Delta$ for $\Neps := T > 0$ which solves the following
   \begin{equation}\label{eq:implicit-HS-solve}
          \sum_{j=1}^m c_j \frac{((\lambda_j^\star+r){L}_j)^{\frac{2}{1+p_j}}D_x^2}{\eps^{\frac{2}{1+p_j}}}T^{-\frac{1+3p_j}{1+p_j}} = 1 \ ,
   \end{equation} 
    with $c_j = 24\left[\frac{1-p_j}{1+p_j}m\right]^\frac{1-p_j}{1+p_j}$.
    Consequently, the convergence rate is at most the sum of the rates of individual terms~\labelcref{eq:optimal-holderSmooth-rate}, weighted by the appropriate multiplier, \begin{equation}\label{eq:HS-general-sum-of-rates}
        \Neps = {\bigO}\left(\sum_{j=1}^m K_{SM}(\eps,(\lambda_j^\star+r)L_j,p_j,D_x)\right) \ .
    \end{equation}
\end{corollary}

\begin{proof}
    Considering the gradient oracle complexity bound $\Neps = \sqrt{24\Lada D_x^2/\eps}$, we substitute $\Lada = \Neps^2 \cdot \eps/(24D_x^2)$ into the definition \labelcref{eq:Lada-definition-general} giving
    $$    \frac{T^2\eps}{24D_x^2}= \sum_{j=1}^m \left[\frac{1-p_j}{1+p_j}\cdot \frac{2\sqrt{6}m\sqrt{ \frac{T^2\eps}{24D_x^2}D_x^2}}{\eps^{3/2}}\right]^\frac{1-p_j}{1+p_j}\left((\lambda_j^\star+r) L_j\right)^\frac{2}{1+p_j} \ .$$
    Rearranging the expression above yields \labelcref{eq:implicit-HS-solve}, which is nonincreasing in $T$. Therefore, to prove~\eqref{eq:HS-general-sum-of-rates}, it suffices to bound each summand of \labelcref{eq:implicit-HS-solve} by $1/m$. We then consider solving $$ c_j \frac{\left((\lambda_j^\star+r)L_j\right)^\frac{2}{1+p_j}D_x^2}{\eps^{\frac{2}{1+p_j}}}T_j^{-\frac{1+3p_j}{1+p_j}}=\frac{1}{m} $$  for $T_j$. The result recovers \labelcref{eq:optimal-holderSmooth-rate} component-wise $$T_j = {\bigO}\left(\left(\frac{(\lambda_j^\star+r)L_j}{\eps}\right)^\frac{2}{1+3p_j}D_x^{\frac{2+2p_j}{1+3p_j}}\right) = {\bigO}\left(K_{SM}(\eps,(\lambda_j^\star+r)L_j,p_j,D_x)\right) . $$ Bounding $T  \leq \max_j T_j \leq \sum_{j=1}^m T_j$, yields \labelcref{eq:HS-general-sum-of-rates}.
\end{proof}

Fixing $h(z) = \sum_{j=1}^m z_j$, each $\lambda_j^\star=1$, and this second corollary recovers the results for heterogeneous sums of \Holder smooth terms of~\citep[Theorem 1.1 and 1.3]{grimmer2023optimal} when we initialize $\lambda_j^0=1$. For constrained optimization, with $h$ as the nonpositive indicator function and when each $g_j$ is $(L_j,p)$-\Holder smooth with common exponent, this second corollary recovers the ${\bigO}(1/\eps^{2/(1+3p)})$ results of~\citep[Corollary 2.3]{deng2024uniformlyoptimalparameterfreefirstorder} as a special case.

{\color{blue} \begin{remark}\label{rmk:tradeoff-lada-dist}
   We note that one may tradeoff knowledge of $\Lada$ for knowledge of distance bounds $D_x$ and $D_\lambda$. Considering the sequence $L_k = 2^0, 2^1, 2^2, ...$, as well as the setting of \cref{thm:heterogeneous-composite-convergence} one may run UFCM for $N_k = \sqrt{\frac{24 L_k D_x^2}{\eps}}$ outer loop iterations. Our theory then guarantees that once $k=\max\{\lceil\log_2(\Lada)\rceil,0\}$, an $\eps$-optimal solution has been constructed, taking at most $\bigO\left(\sqrt{\frac{\Lada D_x^2}{\eps}}\right)$ gradient oracle calls total. Note that neither UFCM nor this modified scheme possesses stopping criteria certifying that an $\epsilon$-optimal solution has been found without additional problem knowledge. So these guarantees are theoretical. 
\end{remark}}

\subsection{Analysis of UFCM for Compositions with Heterogeneous Components} 
The same process of analysis presented in Section~\ref{section:Smooth-Composite} extends to provide guarantees for UFCM given any heterogeneously H\"older smooth components by carefully accounting for the additive errors incurred by using Nesterov-style inequalities. Lemma~\ref{lemma:zhang-smoothness-growth-lemma}, Lemma~\ref{lemma:nu-convergence-bound}, and Proposition~\ref{prop:Q-convergence-bound} generalize to this setting as follows. For many of these results, the proof is redundant with prior work except for tracking an additional constant term through the developed inequalities. Below, we present these key results with proofs of Lemmas~\ref{lemma:Holder-smoothness-growth} and~\ref{lemma:nu-convergence-Holder} deferred to \cref{section:appendix-heterogeneous-comp-proofs} for the sake of completeness.

\begin{lemma}\label{lemma:Holder-smoothness-growth}
    If each $g_j$ is $(L_j,p_j)$-\Holder smooth, then for any fixed $\delta > 0$, \begin{align*}
        \inner{\lambda}{U_{g^*}(\nu;\hat{\nu})} \geq \frac{1}{2L_{\delta,r}}\left\|\sum_{j=1}^m \lambda_j(\nu_j-\hat{\nu_j})\right\|^2-\frac{\delta}{2}, \ \forall \lambda \in \Lambda_r, \ \forall \nu, \hat{\nu} \in \{\nabla g(x) : x \in \X\}  \ .      
    \end{align*}
   
\end{lemma}

The following results utilize the general smoothness constant $L_{\delta,r}$ in both the analysis and the appropriate parameters for completeness. However, we recall that the Approximate Aggregate Smoothness constant $\Lada$ used in \cref{thm:heterogeneous-composite-convergence} is specialized with $\delta = \eps/\sqrt{24\Lada D_x^2/\eps}$. 
\begin{lemma}\label{lemma:nu-convergence-Holder} 
Suppose the stepsizes satisfy~\eqref{cond:a1}-\eqref{cond:c3}, and let $z^t := (x^t; \Tilde{\lambda}^t, \nu^t)$ denote the iterates of Algorithm 1. Then $z^t$ satisfy the following for any $z=(x;\lambda,\nu) \in \X \times \Lambda_r \times V$ and any fixed $\delta > 0$
\begin{align*}
\begin{split}        
            \sum_{t=1}^T \omega_t\left[\hQ_\nu(z^t,z)\right] &\leq \frac{\omega_T L_{\delta,r}}{2(\tau_T+1)}\|x^T-x^{T-1}\|^2+\sum_{t=1}^{T-1}\frac{\omega_t\theta_{t+1}L_{\delta,r}}{2\tau_{t+1}}\|x^t-x^{t-1}\|^2\\&\hspace{0cm}+\omega_1\tau_1\inner{\lambda}{U_{g^*}(\nu;\nu^0)}+\frac{\delta}{2}\left[\omega_T(\tau_T+1)+\sum_{t=2}^T \omega_t\tau_t\right], \quad \forall (x;\lambda,\nu) \in \X \times \Lambda_r \times V \ .
\end{split}
\end{align*}
\end{lemma}
The following proposition is the direct analog of \cref{prop:Q-convergence-bound} in the heterogeneously smooth setting. The proof is analogous, applying the above two lemmas instead of Lemmas~\ref{lemma:zhang-smoothness-growth-lemma} and~\ref{lemma:nu-convergence-bound}, noting the small additive dependence on tolerance $\delta$.
\begin{proposition}\label{prop:general-Holder-Q-bound} Consider any problem of the form \labelcref{eq:WLOG-main-problem} with stepsizes satisfying \labelcref{cond:a1}-\labelcref{cond:c3}. Then for any $z = \displaystyle(x; \lambda, \nu) \in \X \times \Lambda_r \times V$ and any fixed $\delta > 0$, \begin{align*}
   &\sum_{t=1}^T \omega_t\hQ(z^t,z){+\color{blue}\sum_{t=1}^T\sum_{s=1}^{S_t}\frac{\omega_t}{S_t}\muhstar\|\lambda_s^{(t)}-\lambda\|^2}+\frac{\omega_T\eta_T}{2}\|x^T-x\|^2\\ &\leq \frac{\Tilde{\omega}^{(1)}\beta^{(1)}+\omega_1\eta_1}{2}\|x^0-x\|^2+\frac{\Tilde{\omega}^{(1)}\gamma^{(1)}}{2}\|\lambda^0-\lambda\|^2 \hspace{0cm}+\omega_1\tau_1\inner{\lambda}{U_{g^*}(\nu;\nu^0)}\\&\hspace{1cm}+\frac{\delta}{2}\left[\omega_T(\tau_T+1)+\sum_{t=2}^T \omega_t\tau_t\right], \quad \forall (x;\lambda,\nu) \in \X \times \Lambda_r \times V  \ .
\end{align*}
\end{proposition}
\paragraph{Proof of \cref{thm:heterogeneous-composite-convergence}.}
 We first note that as $\omega_t=t$, $\tau_t=\frac{t-1}{2}$, we can rewrite $$\left[\omega_T(\tau_T+1)+\sum_{t=2}^T\omega_t\tau_t\right]=\frac{T^3+3T^2+2T}{6}, \quad \sum_{t=1}^T \omega_t = \frac{T(T+1)}{2}\ .$$ Then for all $\lambda \in \Lambda_r$ and $\nu \in V$,
 considering \cref{prop:general-Holder-Q-bound}, Jensen's inequality \labelcref{eq:Gap-Jensen's}, and the particular stepsizes \labelcref{eq:UFCM-stepsize-t=1}, we can bound \begin{equation}\label{eq:Q-bound-Holder}
     \hQ(\bar{z}^T, (x^\star; \lambda, \nu)) \leq \frac{(C/\Delta+2L_{\delta,r})\|x^0-x^\star\|^2+2/(C\Delta)(\|\lambda^0-\lambda^\star\|^2+r^2)}{T(T+1)}+\frac{\delta (T+2)}{6} \ .
 \end{equation} 
 Recall that $L_{\delta,r}$ depends on the value $\delta$, so we optimize the above bound with respect $\delta$ to achieve a universally optimal rate.
 We fix $T, \eps > 0$.
        
        Setting $\delta=\frac{\eps}{T}$, and bounding $T+2 \leq 3T$, we obtain the following inequality derived from \labelcref{eq:Q-bound-Holder} \begin{equation}\label{eq:Heterogeneous-Q-bound-eps/N}
            \hQ(\bar{z}^T, (x^\star; \lambda, \nu)) \leq \frac{(C/\Delta+2L_{\eps/T,r})\|x^0-x^\star\|^2+2/(C\Delta)(\|\lambda^0-\lambda^\star\|^2+r^2)}{T^2}+\frac{\eps}{2} \ .\end{equation}

            Our choices of $C$, $\Delta$, and $r=D_\lambda\sqrt{\eps}$ simplify the expression as for any $T \geq \Neps = \sqrt{24\Lada D_x^2/\eps}$ one has $Q(\bar{z}^T , (x^\star;\lambda,\nu)) \leq \eps$ over all $\lambda \in \Lambda_r$ and $\nu \in V$. Applying \cref{lemma:Q-small->eps-opt}, this ensures $(\epsilon,r)$-optimality. Furthermore, when $T = \Neps$, then $L_{\eps/\Neps,r}$ precisely recovers the definition of $\Lada$ in \labelcref{eq:Lada-definition-general}.
            
            The claimed proximal step complexity follows from the general formula~\labelcref{eq:general-P_eps}. Since $\eps \leq 24\Lada D_x^2$, it holds that $\Neps \geq 1$. Therefore, we bound $\lceil \Neps \rceil \leq \Neps + 1 \leq  2\Neps$, which results in $$\Peps = \lceil \Neps \rceil + \lceil \Neps \rceil^2\Delta M \leq \sqrt{\frac{24\Lada D_x^2}{\eps}}+\frac{48MD_xD_\lambda}{\eps} + 1 \ ,$$
            where we recall $\Delta = D_\lambda/(2D_x\Lada)$.

\section{Growth Bounds and Restarting}\label{section:UC-restarting-comp}
We utilize a simple restarting scheme given initial distance bounds. Primal-dual algorithms have exhibited great success from restarting when the respective gap function possesses certain growth conditions~\citep{lu_LP_sharp,lu2024practicaloptimalfirstordermethod,Fercoq_restarting}. This algorithm, denoted R-UFCM can then achieve linear convergence in terms of gradient oracle calls when the components are smooth and strongly convex, and the proximal step complexity can achieve linear convergence rates when $h$ is sufficiently smooth. Recall from \labelcref{eq:intro-UC-def} that a function $f$ is $(\mu,q)$-uniformly convex if $$f(y) \geq f(x)+\inner{\nabla f(x)}{y-x}+\frac{\mu}{q+1}\|y-x\|^{q+1}, \quad\forall x,y \in \dom{f} \ .$$ 

 When $q=1$, we recover the notion of $\mu$-strong convexity. As $q \to \infty$, these functions are simply convex. Similarly to \Holder smoothness, we can interpolate between the level of convexity. If $g_j$ is also $(L_j, p_j)$-\Holder smooth, then the following symmetric, two-sided bound holds $$ g_j(x)+\inner{\nabla g_j(x)}{y-x}+\frac{\mu_j}{q_j+1}\|y-x\|^{q_j+1} \leq g_j(y) \leq  g_j(x)+\inner{\nabla g_j(x)}{y-x}+\frac{L_{j}}{p_j+1}\|y-x\|^{p_j+1}\ .$$

\subsection{Growth Structure}
The uniform convexity of each $g_j$ can be combined together to ensure a growth condition on the gap function. This perspective plays a central role in our analysis, as it does in most restarted analyses.
\begin{definition}\label{def:duality-gap_growth}
Given monotone nondecreasing, convex functions  $G_x, G_\lambda \colon \R_+ \to \R_+$, we say that the gap function possesses $(G_x, G_\lambda)$-growth if for any $z = (x; \lambda,\nu)$ and $\hat{z}=(x^\star;\lambda^\star, \nabla g(x))$ $$G_x(\|x-x^\star\|)+G_\lambda(\|\lambda-\lambda^\star\|) \leq Q(z,\hat z) \ .$$
\end{definition} 
As additional structure, the growth functions considered herein will always have $G_x(0)=0$, $G_\lambda(0)=0$, and both $G_x$ and $G_\lambda$ differentiable. The following lemma gives the explicit growth condition when the component functions $g_j$ exhibit varying uniform convexity and $h$ is $L_h$-smooth.

    \begin{lemma}\label{lemma:Hetero-growth-x} 
        Suppose component functions $g_j$ are $(\mu_j,q_j)$-uniformly convex and $h$ is $L_h$-smooth. Then the gap function possesses $G_x,G_\lambda$ growth where $G_x(t) = \sum_{j=1}^m \lambda_j^\star\frac{\mu_j}{q_j+1}|t|^{(q_j+1)}$ and $G_\lambda(t)=\frac{1}{2L_h}t^{2}$. Therefore, for any $z=(x,\lambda,\nu)$ and $\hat{z}=(x^\star,\lambda^\star,\nabla g(x))$ $$Q(z,\hat z) \geq \sum_{j=1}^m \lambda_j^\star \frac{\mu_j}{q_j+1}\|x-x^\star\|^{q_j+1} + \frac{1}{2L_h}\|\lambda-\lambda^\star\|^{2}\ .$$
    \end{lemma}

    \begin{proof}
From the optimality of $x^\star$, $\inner{\nabla u(x^\star)+\sum_{j=1}^m \lambda_j^\star \nabla g_j(x^\star)}{x-x^\star} \geq 0$. Note that since $h$ is $L_h$-smooth, $h^*$ is $1/L_h$-strongly convex. Thus, 
        \begin{align*}
            Q(z,\hat{z})  &= \Lcal(x; \lambda^\star,  \nabla g(x))-\Lcal(x^\star; \lambda^\star,  \nu^\star)+\Lcal(x^\star; \lambda^\star, \nu^\star)-\Lcal(x^\star; \lambda,  \nu)\\
            &\geq u(x)+\inner{\lambda^\star}{g(x)}-h^*(\lambda^\star)-[u(x^\star)+\inner{\lambda^\star}{g(x^\star)}-h^*(\lambda^\star)]\\&\hspace{0.8cm}+\inner{\lambda^\star}{\nu^\star x^\star-g^*(\nu^\star)}-h^*(\lambda^\star)-[\inner{\lambda}{\nu x^\star-g^*(\nu)}-h^*(\lambda)]\\
            &\hspace{0.8cm}-\inner{\nabla u(x^\star)+\sum_{j=1}^m \lambda_j^\star \nabla g_j(x^\star)}{x-x^\star}\\
            &\geq \left[u({x})-u(x^\star)-\inner{\nabla u(x^\star)}{x-x^\star}\right]+\sum_{j=1}^m \lambda_j^\star\left[g(x)-g(x^\star)-\inner{\nabla g(x^\star)}{x-x^\star}\right]\\
            &\hspace{0.8cm} +h^*(\lambda)-h^*(\lambda^\star)-\inner{\lambda-\lambda^\star}{g(x^\star)}\\
            &\geq \sum_{j=1}^m \lambda_j^\star \frac{\mu_j}{q_j+1}\|x-x^\star\|^{q_j+1} +\frac{1}{2L_h}\|\lambda-\lambda^\star\|^{2}\ ,
        \end{align*}
        where the first equality expands the gap function, the following inequality applies Fenchel-Young and subtracts the nonnegative inner product outlined above, the next inequality regroups terms and applies Fenchel-Young once again, and the final inequality comes directly from the convexity of $u$, the uniform convexity of $g_j$, and the strong convexity of $h^*$.
    \end{proof}

    \subsection{An Approximate Dualized Aggregate  Convexity \texorpdfstring{$\muada$}{Lg}}\label{subsec:muada}
 We now have the necessary tools to define the lower bounding curvature for the composite problem into a single value $\muada$, generalizing the growth bound strong convexity yields. For $(\mu_j,q_j)$-uniformly convex components $g_j$ and target accuracy $\eps > 0$, we define the Approximate Dualized Aggregate convexity constant implicitly as the unique positive solution to the following equation
\begin{equation}\label{eq:muada-def-general}
    {\muada} := \left\{\mu^{\mathtt{ADA}} > 0 : \frac{\mu^{\mathtt{ADA}}}{2} = \sum_{j=1}^m \lambda^\star_j\frac{\mu_j}{q_j+1}(\eps/\mu^{\mathtt{ADA}})^\frac{q_j-1}{2}\right\} \ .
\end{equation}  
Note when $q_j=1$, the coefficient becomes independent of $\eps$. If all $q_j = 1$, the $\muada$ simply totals the $\lambda_j^\star$-weighted strong convexity constants. More generally, $\muada$ aggregates the lower curvature of each component, weighted by the appropriate dual multiplier. This quantity can further be viewed as an approximation of strong convexity as shown in Lemma~\ref{lemma:UC-quad-lower-bound} below. 
 \begin{lemma}\label{lemma:muada-nondecreasing}
     The Approximate Dualized Aggregate convexity constant $\muada$ as defined in \labelcref{eq:muada-def-general} is nondecreasing with respect to $\eps$.
 \end{lemma}
 \begin{proof}
 Consider $\eps' \geq \eps > 0$. Rearranging the definitions of $\muada$ and $\mu_{\eps'}^\mathtt{ADA}$ ensure that $$\sum_{j=1}^m \lambda_j^\star \frac{\mu_j}{q_j+1}(\muada)^{-\frac{q_j+1}{2}}\eps^\frac{q_j-1}{2} = 1, \quad \sum_{j=1}^m \lambda_j^\star \frac{\mu_j}{q_j+1}(\mu_{\eps'}^\mathtt{ADA})^{-\frac{q_j+1}{2}}(\eps')^\frac{q_j-1}{2} =  1 \ .$$
  Since each $q_j \geq 1$, it follows that $(\eps')^{(q_j-1)/2} \geq (\eps)^{(q_j-1)/2}$, and for the above sums to equal one, it must hold that the positive solution $\mu_{\eps'}^\mathtt{ADA} \geq \muada$.
 \end{proof}
 \begin{lemma}\label{lemma:UC-quad-lower-bound}
     Suppose the components $g_j$ are $(\mu_j,q_j)$-uniformly convex and $h$ is $L_h$-smooth. Then  for any $z=(x,\lambda,\nu)$ and $\hat{z}=(x^\star,\lambda^\star,\nabla g(x))$,  $$Q(z,\hat z) \geq \frac{\muada}{2}\|x-x^\star\|^2+\frac{1}{2L_h}\|\lambda-\lambda^\star\|^2-\frac{\eps}{2} \ .$$
 \end{lemma}
 \begin{proof}
 Considering the result of \cref{lemma:Hetero-growth-x}, it suffices to bound $$G_x(\|x-x^\star\|) \geq \frac{\muada}{2}\|x-x^\star\|^2-\frac{\eps}{2}, \quad \forall x \in \X \ .$$
 Note that $\muada = \eps/(G_x^{-1}(\eps/2))^2$. Since $G_x(t)$ is differentiable and positive for all $t > 0$ and  $\muada$ is nondecreasing in $\eps$, it follows that
     $$\frac{\partial \muada}{\partial \eps}=\frac{(G_x^{-1}(\eps/2))^2 - \frac{\eps G_x^{-1}(\eps/2)}{G_x'(G_x^{-1}(\eps/2))}}{(G_x^{-1}(\eps/2))^4} \geq 0 \Longrightarrow G_x'(G_x^{-1}(\eps/2)) \geq \frac{\eps}{G_x^{-1}(\eps/2)}=\muada G_x^{-1}(\eps/2) \ .$$
     Therefore, by the monotonicity and nonnegativity of $G_x$, for any $t \geq G_x^{-1}(\eps/2)$, it holds that $G_x'(t) \geq \muada t$.
     
     We first consider the case where $G_x(\|x-x^\star\|) \geq \eps/2$. We again note that by monotonicity and nonegativity of $G_x$, it holds that for any $x \in \X$, \begin{align*}\label{eq:muada-exact_quad_lower_boud}G_x(\|x-x^*\|) = \int_0^{\|x-x^\star\|} G_x'(t)\mathrm{dt} &\geq \int_{G_x^{-1}(\eps/2)}^{\|x-x^\star\|} \muada t \mathrm{dt}+\int_0^{G_x^{-1}(\eps/2)}G_x'(t)\mathrm{dt} \nonumber \\
     &= \int_{0}^{\|x-x^\star\|} \muada t \mathrm{dt} - \int_0^{G_x^{-1}(\eps/2)} \muada t \mathrm{dt} + \int_0^{G_x^{-1}(\eps/2)}G_x'(t)\mathrm{dt} \nonumber\\
    &=\frac{\muada}{2}\|x-x^\star\|^2 \ .
     \end{align*}  

    Now we consider the case where $G_x(\|x-x^\star\|) < \eps/2$. Since $\muada = \eps/(G_x^{-1}(\eps/2))^2$, we note that $$\frac{\muada}{2}\|x-x^\star\|^2 = \frac{\eps}{2(G_x^{-1}(\eps/2))^2}\|x-x^\star\|^2 < \frac{\eps}{2} \ ,$$
    which implies that 
    \begin{equation*}
    G_x(\|x-x^\star\|) \geq 0 > \frac{\muada}{2}\|x-x^\star\|^2 - \frac{\eps}{2} \ . \qedhere
    \end{equation*}
 \end{proof}
   
\subsection{A Further Universalized Approximate Dualized Aggregate Smoothness}
Recall that the previous definition in \labelcref{eq:Lada-definition-general} for the Approximate Dualized Aggregate smoothness constant $\Lada$ depended on $\eps$ and distance bound $D_x$. In order to recover \labelcref{eq:optimal-holderSmooth-rate} through \cref{thm:heterogeneous-composite-convergence}, this dependence was a necessity. When the components possess uniform convexity in addition to \Holder smoothness, one can further leverage the Approximate Dualized Aggregate Convexity $\muada$. In its full generality, we define $\Lada$ to be the unique positive solution to the following equation
\begin{equation} \label{eq:Lada-definition-fully-general}
   \Lada := \left\{L^\mathtt{ADA} > 0 : L^\mathtt{ADA} = \sum_{j=1}^m \left[\frac{1-p_j}{1+p_j}\cdot \frac{m\sqrt{L^\mathtt{ADA}}}{\eps}\cdot \min\left\{\frac{2\sqrt{6}D_x}{\sqrt{\eps}}, \frac{4\sqrt{6}}{\sqrt{\muada}}\right\}\right]^\frac{1-p_j}{1+p_j}\left[(\lambda_j^\star+r) L_j\right]^\frac{2}{1+p_j}\right\} \ .
\end{equation}
We note that for small enough $\muada$, the above value recovers \labelcref{eq:Lada-definition-general} exactly. Further note that even with this generalization, $\Lada$ remains nonincreasing with respect to $\eps$.

\subsection{Guarantees for Fully Heterogeneous Compositions}
Finally, we present our universal theory when each component $g_j$ possesses its own $(L_j,p_j)$-\Holder smoothness and $(\mu_j,q_j)$-uniform convexity. Algorithmic restarting, as discussed in Section~\ref{subsubsec:restarting}, is the key to enabling this final improvement in our theory.

Our proposed restarted variant, denoted R-UFCM, {\color{blue} sequentially runs $K$ executions of} UFCM, each for $T_k$ iterations, restarted at a sequence of initializations $z^k = (x^k, \lambda^k)$ with distance bounds $D_x^{(k)}$ and $D_\lambda^{(k)}$. Using the produced outputs $\bar{x}^{T_k,k}$ and $\bar{\lambda}^{T_k,k}$, the next initialization $z^{k+1}= (x^{k+1}, \lambda^{k+1})$ is determined. The next primal initialization is $\bar x^{T_k,k}$ if $\muada \geq 4\eps/D_x^2$, else $x^0$ is reused. Similarly, the next dual initialization is $\bar \lambda^{T_k,k}$ if {\color{blue} $2^{K-k}L_h \leq D_\lambda^2/\eps$}, else $\lambda^0$ is reused. \Cref{alg:R-FCM} formalizes this process with the following initializations

\begin{align}\label{eq:R-UFCM-initializations}
   (T_k,D_x^{(0)}) &= \begin{cases}
               \left(\sqrt{\frac{{\color{blue}96} \Lada}{\muada}},\sqrt{\frac{2^{K+1}\eps}{\muada}}\right) & \text{if } \muada \geq \frac{4\eps}{D_x^2}, \\
               \left(\sqrt{\frac{{\color{blue}24}\Lada D_x^2}{2^{K-k-1}\eps}},D_x\right) & \text{otherwise,}
          \end{cases} \quad
   D_\lambda^{(0)} = {\color{blue} \min\left\{D_\lambda, \sqrt{2^{K+1}\eps L_h}\right\} }  
\end{align}
Note that when $\muada \geq 4\eps/D_x^2$, $T_k$ is independent of $k$.

\begin{algorithm}[t]
\textbf{Input}  $z^{0} \in \X \times \Lambda, \text{ distance bounds } D_x \text{ and } D_\lambda$, target accuracy $\eps > 0$, constants $\Lada$ and $\muada$, and UFCM {\color{blue} execution count  $K$}
\begin{algorithmic}[1]
\State Set $D_x^{(0)}$, $D_\lambda^{(0)}$ and $\{T_k\}$ according to \labelcref{eq:R-UFCM-initializations}
\For{$k = 0, 1, \dots, K-1$}
   \State Run UFCM($z^k$, $\lceil T_k \rceil$, $\Lada$) returning output $(\bar{x}^{T_k,k},\bar{\lambda}^{T_k,k})$ 
   \State Set $(x^{k+1}, D_x^{(k+1)})= \begin{cases}
       (\bar{x}^{T_k,k},\sqrt{2^{K-k}\eps/\muada}) &\text{if } \muada \geq 4\eps/D_x^2\\
       (x^0,D_x) & \text{otherwise}
   \end{cases}$
      \State Set $(\lambda^{k+1}, D_\lambda^{(k+1)})= \begin{cases}
   (\bar{\lambda}^{T_k,k},\sqrt{2^{K-k}\eps L_h})&  \text{if }{\color{blue} \sqrt{2^{K-k}\eps L_h}} \leq D_\lambda\\
       (\lambda^0,D_\lambda) & \text{otherwise}
   \end{cases}$
    \State Set $z^{k+1} = (x^{k+1},\lambda^{k+1})$ 
\EndFor
\end{algorithmic}
\caption{Restarted Universal Fast Composite Method (R-UFCM)}\label{alg:R-FCM}
\end{algorithm}

The following theorem, proven in \cref{section:proof-of-RUFCM}, establishes our universal convergence theory. We denote the gradient complexity of this restarted method by $\Neps := \sum_{k=0}^{K-1}\lceil T_k\rceil$ as R-UFCM computes $\lceil T_k\rceil$ gradients of each $g_j$ in execution $k$ of UFCM. Likewise, we denote the proximal complexity by $\Peps := \sum_{k=0}^{K-1} \lceil \Peps^{(k)} \rceil$ where $\lceil \Peps^{(k)}\rceil$ bounds the number of proximal evaluations of $u$ and $h$ used in the $k$th execution of UFCM. 

Furthermore, we restrict $\eps \in (0,1]$ sufficiently small such that $$\sqrt{\frac{24 \Lada D_x^2}{\eps}} \geq 1, \quad \sqrt{\frac{96\Lada}{\muada}} \geq 1 .$$
These restrictions must hold for sufficiently small $\eps$ as $\lim_{\eps \to 0^+} {\Lada}/{\eps} = +\infty$, which holds from \cref{lemma:Lada-nonincreasing}.
Secondly, $\lim_{\eps \to 0^+} {\Lada}/{\muada} \geq \max\left\{\lim_{\eps \to 0^+}{L_{1,r}^\mathtt{ADA}}/{\muada},\lim_{\eps \to 0^+}{\Lada}/{\mu_1^\mathtt{ADA}} \right\} \geq 1$, where the first inequality utilizes Lemmas~\ref{lemma:Lada-nonincreasing} and~\ref{lemma:muada-nondecreasing}, and the second notes that when $p_j=q_j=1$ for all $j$, then $\Lada$ and $\muada$ are constant with respect to $\eps$, so $\Lada \geq \muada$, while if any $p_j < 1$ or $q_j > 1$ then the first or second limit diverge to infinity respectively. 

For notational ease, we let $\Tilde{z}^k = (x^k; \lambda^k, \nabla g(x^k))$, extending each initialization $z^k = (x^k,\lambda^k)$ to include the conjugate variables. {\color{blue} In particular $\Tilde{z}^0 = (x^0; \lambda^0, \nabla g(x^0))$. }
{\color{blue} We also let $\hat{z}^0 = (x^\star; \lambda^\star, \nabla g(x^0))$, extending the optimal primal-dual pair to include the conjugate variable at the initialization.}

\begin{theorem}\label{thm:fully-heteogeneous-component-guarantees}
    Consider any problem of the form~\labelcref{eq:WLOG-main-problem} with each $g_j$ being $(L_j,p_j)$-H\"older smooth and $(\mu_j,q_j)$-uniformly convex, target accuracy $\eps > 0$ sufficiently small, with $r=D_\lambda\sqrt{\eps}$.  Setting $C^{(k)}=D_\lambda^{(k)}/D_x^{(k)}$ and $\Delta^{(k)}=C^{(k)}/(2\Lada)$, {\color{blue} if $K \geq \left\lceil\log_2\left(\frac{Q(\Tilde z^0,\hat{z}^0)+\eps}{\eps}\right)\right\rceil$}, \Cref{alg:R-FCM} with stepsizes \labelcref{eq:outerloop-stepsizes} and \labelcref{eq:innerloop-stepsizes} must find an $(\eps,r)$-optimal solution~\labelcref{def:optimality-conditions}. {\color{blue} If $K$ is within a constant factor of $\left\lceil\log_2\left(\frac{Q(\Tilde z^0,\hat{z}^0)+\eps}{\eps}\right)\right\rceil$, this achieves the complexity bounds outlined in \cref{tab:composite-heterogeneous-rates}.}
\end{theorem}

\begin{remark}\label{rmk:smooth-sc-analog}
    Since $L_{\eps,r}^{\mathtt{ADA}}$ is nonincreasing with $\eps$ and $\mu_{\eps}^{\mathtt{ADA}}$ is nondecreasing with $\eps$, this bound can be tightened by considering our \textit{Approximate Dualized Aggregate} constants specialized to the target accuracy sought by each application of UFCM. For each loop, one could run $UFCM(z^k,\lceil T_k \rceil, L_{2^{K-k-1}\eps,r}^\mathtt{ADA})$ with outer loop iteration count $$T_k = \min\left\{{\sqrt{\frac{96L_{2^{K-k-1}\eps,r}^\mathtt{ADA}}{\mu_{2^{K-k-1}\eps}^\mathtt{ADA}}}},\sqrt{\frac{24L_{2^{K-k-1}\eps,r}^\mathtt{ADA}D_x^2}{2^{K-k-1}\eps}}\right\} \ ,$$ 
    instead updating $D_x^{(k+1)}$ with $\sqrt{2^{K-k}\eps/\mu_{2^{K-k}\eps}^\mathtt{ADA}}$  whenever $\mu_{2^{K-k} \eps}^\mathtt{ADA} \geq {2^{K-k+2}\eps}/D_x^2$ and $D_\lambda^{(k+1)}$ with $\sqrt{2^{K-k}\eps L_h}$ whenever $L_h \leq D_\lambda^2/(2^{K-k}\eps)$. Consequently, one can derive guarantees $$\Neps = {\bigO}\left(\sum_{n=0}^{K-1} \sqrt{\frac{L_{2^{K-k-1}\eps,r}^\mathtt{ADA}}{\mu_{2^{K-k-1}\eps}^\mathtt{ADA}}}\right) , \quad \Peps = {\bigO}\left(\sum_{n=0}^{K-1}\sqrt{\frac{(L_{2^{K-k-1}\eps,r}^\mathtt{ADA}+M^2L_h)}{\mu_{2^{K-k-1}\eps}^\mathtt{ADA}}}\right)$$ which avoids additional multiplicative log terms if the sums above total up geometrically. 
\end{remark}

\begin{corollary}\label{cor:muada-recover-UC-bounds}
   Consider minimizing $F(x) = g_0(x)+u(x)$ where $g_0$ is $(L,p)$-\Holder smooth and $(\mu,q)$-uniformly convex function with $D_x \geq \|x^0-x^\star\|$, initialization $\lambda^0= 1$, and any target accuracy $0 < \eps \leq \min\left\{1,2\sqrt{6}LD_x^{1+p}, \frac{2\mu}{1+q}\left(\frac{D_x}{2}\right)^{1+q},4\left(\frac{1+q}{2}\right)^{\frac{2}{3q+1}} L^{\frac{2\left(q+1\right)}{\left(3q+1\right)\left(1+p\right)}}\right\}$. Then \Cref{alg:R-FCM} recovers \labelcref{eq:optimal-holderSmoothUniformlyConvex-rate}:
   $$\Neps = \Peps =  K_{UC}(\eps,L,p,\mu,q) = \begin{cases}{\bigO}\left(\left(\frac{L^{1+q}}{\mu^{1+p}\eps^{q-p}}\right)^{\frac{2}{(1+3p)(1+q)}}\right) &\quad \text{if } q>p \ , \\
    {\bigO}\left(\left(\frac{L^{1+q}}{\mu^{1+p}}\right)^{\frac{2}{(1+3p)(1+q)}}\log\left(\frac{F(x^0)-F^\star}{\eps}\right)\right)  &\quad \text{if }q=p
    \end{cases}
$$
up to logarithmic factors\footnote{Using the modification discussed in  
 \cref{rmk:smooth-sc-analog}, one can recover the optimal rate without incurring log factors.}.
\end{corollary}

\begin{proof}

By hypothesis, $\eps$ is sufficiently small to apply \cref{thm:fully-heteogeneous-component-guarantees}. Noting our Approximate Dualized Aggregate constants equal
$$\Lada = (1+r)^\frac{4}{1+3p}\left[\frac{1-p}{1+p}\cdot \frac{4\sqrt{6}}{\eps\sqrt{\muada}}\right]^\frac{2-2p}{1+3p} L^\frac{4}{1+3p}  \quad \text{and } \quad \muada =  \left(\frac{2\mu}{1+q}\right)^\frac{2}{1+q}\eps^\frac{q-1}{q+1} \ ,$$
we conclude 
\begin{align*}
       \Neps &= \LogBigO\left(\sqrt{\frac{\Lada}{\muada}}\right) =  \LogBigO\left({\frac{(\eps \sqrt{\muada})^{-\frac{1-p}{1+3p}}L^\frac{2}{1+3p}}{\sqrt{\muada}}}\right) =   \LogBigO\left(\left(\frac{L^{1+q}}{\mu^{1+p}\eps^{q-p}}\right)^{\frac{2}{(1+3p)(1+q)}}\right) \ ,
   \end{align*}
where the first equality considers  \cref{thm:fully-heteogeneous-component-guarantees},   the second equality substitutes $\Lada$, and the last equality then substitutes $\muada$ and simplifies to recover \labelcref{eq:optimal-holderSmoothUniformlyConvex-rate}. 

Since $\lambda^0=\lambda^\star$, we can make $\Delta$ arbitrarily small, so $S_t = 1$ for each $t$. Therefore, the resulting proximal complexity equals the gradient oracle complexity. 
\end{proof}
 
\begin{corollary}\label{cor:UC-sum-bound}
    For any problem of the form \labelcref{eq:WLOG-main-problem}, target accuracy $\eps>0$ sufficiently small, with $r = D_\lambda \sqrt{\eps}$, suppose each $g_j$ is $(L_j,p_j)$-smooth and $(\mu_j,q_j)$-uniformly convex. \Cref{alg:R-FCM}, with stepsizes~\eqref{eq:outerloop-stepsizes} and~\eqref{eq:innerloop-stepsizes}, with choices of $C^{(k)}=D_\lambda^{(k)}/D_x^{(k)}$ and $\Delta^{(k)}=C^{(k)}/(2\Lada)$ must find an $(\eps,r)$-optimal solution with oracle complexity bound $$\Neps  = \LogBigO\left(\sum_{j=1}^m K_{UC}(\eps,(\lambda_j^\star+r)L_j,p_j,\muada,1) \right) \ .$$



\end{corollary}

\begin{proof}
Since $\eps \leq D_x^2\muada/4$, it holds that after rearrangement of the definition in \labelcref{eq:Lada-definition-fully-general}, $\Lada$ is the unique positive root to $$ \sum_{j=1}^m (\Lada)^{-\frac{1+3p_j}{2(1+p_j)}}\left[\frac{1-p_j}{1+p_j}\cdot \frac{4\sqrt{6}m}{\eps\sqrt{\muada}}\right]^\frac{1-p_j}{1+p_j}\left[(\lambda_j^\star+r)L_j\right]^\frac{2}{1+p_j} = 1 \ . $$  
Similar to proving \cref{cor:Lada-general-implicit-formula}, we can bound ${\color{blue} \sqrt{\Lada} \leq \sum_{j=1}^m \sqrt{L_{j,\eps,r}^{\mathtt{ADA}}}}$ where $L_{j,\eps,r}^{\mathtt{ADA}}$ solves component-wise as the unique positive root to the following equation $$(L_{j,\eps,r}^\mathtt{ADA})^{-\frac{1+3p_j}{2(1+p_j)}}\left[\frac{1-p_j}{1+p_j}\cdot \frac{4\sqrt{6}m}{\eps\sqrt{\muada}}\right]^\frac{1-p_j}{1+p_j}\left[(\lambda_j^\star+r)L_j\right]^\frac{2}{1+p_j} = \frac{1}{m} \ .$$
 We can then conclude \begin{equation}\label{eq:Lada-heterogeneous-bound}
    \sqrt{\Lada} \leq \sum_{j=1}^m m^\frac{2}{1+3p_j}\left[\frac{1-p_j}{1+p_j}\cdot \frac{4\sqrt{6}}{\eps\sqrt{\muada}}\right]^\frac{1-p_j}{1+3p_j}\left[(\lambda_j^\star+r)L_j\right]^\frac{2}{1+3p_j} \ .
\end{equation}
Finally, it holds that
 $$\Neps =\LogBigO \left(\sum_{j=1}^m \frac{\left(\eps\sqrt{\muada}\right)^{-\frac{1-p_j}{1+3p_j}}\left[(\lambda_j^\star+r)L_j\right]^\frac{2}{1+3p_j}}{\sqrt{\muada}}\right)=\LogBigO\left(\sum_{j=1}^m \left(\frac{\left[(\lambda_j^\star+r)L_j\right]^2}{(\muada)^{1+p_j}\eps^{1-p_j}}\right)^\frac{1}{(1+3p_j)}\right)\ ,$$
where the first equality considers the result from \cref{thm:fully-heteogeneous-component-guarantees} and substitutes the upper bound on $\Lada$ in \labelcref{eq:Lada-heterogeneous-bound}, and the second equality simplifies to yield the desired result.
\end{proof}
Fixing $h(z) = \sum_{j=1}^m z_j$, each $\lambda_j^\star=1$, and bounding $\muada$ by only considering a single component in its sum recovers the results for heterogeneous sums of \Holder smooth terms of~\citep[Theorem 1.2]{grimmer2023optimal}. When each $g_j$ is smooth and $h$ is a nonpositive indicator function, this second corollary recovers the results of~\citep[Theorem 6]{zhang2022solving} by lower bounding $\muada$ by the $\mu_0$-strong convexity of $g_0$ (see the concluding Section~\ref{section:application-to-func-constrained} for further consideration of this special case).

\subsection{Analysis of R-UFCM (Proof of \cref{thm:fully-heteogeneous-component-guarantees})}\label{section:proof-of-RUFCM}

   Recall our analysis only depends on the $(L_j,p_j)$-\Holder smooth and $(\mu_j,q_j)$-uniformly convex of $g_j$ through our analysis through the universal constants $\Lada$ and $\muada$ defined in~\eqref{eq:Lada-definition-fully-general} and~\eqref{eq:muada-def-general}. {\color{blue} Let $\lambda \in \Lambda_r$ and $\nu \in V$.} Below, we inductively prove that in all four of the cases in Table~\ref{tab:composite-heterogeneous-rates} (determined by whether $\muada < 4\eps/D_x^2$ and whether $L_h > D_\lambda^2/\eps$) the following are maintained at each outer iteration of the restarted method $k=0,1,\dots,K-1$
   \begin{equation*}
        D_x^{(k)} \geq \|x^k-x^\star\|, \quad D_\lambda^{(k)} \geq \|\lambda^k-\lambda^\star\|, \quad Q((\bar x^{T_k,k};\bar{\lambda}^{T_k,k}, \bar \nu^{T_k,k}),(x^\star; \lambda,\nu)) \leq 2^{K-k-1}\eps \ , 
   \end{equation*}
   where we recall $(\bar x^{T_k,k};\bar{\lambda}^{T_k,k}, \bar \nu^{T_k,k})$ from our averaging scheme \labelcref{eq:averaging-scheme}.
   {\color{blue} By definition and application of \cref{lemma:UC-quad-lower-bound}, $D_x^{(0)} \geq \|x^0-x^\star\|$ and $D_\lambda^{(0)} \geq \|\lambda^0-\lambda^\star\|$ both hold at $k=0$, regardless of the relative sizes of $\muada$ and $L_h$}. Our inductive proof proceeds by first establishing that
   \begin{equation} \label{eq:restarted-induction-step-one}
       D_x^{(k)} \geq \|x^k-x^\star\|, \quad D_\lambda^{(k)} \geq \|\lambda^k-\lambda^\star\| \implies Q((\bar x^{T_k,k};\bar{\lambda}^{T_k,k}, \bar \nu^{T_k,k}),(x^\star; \lambda,\nu)) \leq 2^{K-k-1}\eps
   \end{equation}
   for each $k$. The key result to this end is that $Q((\bar x^{T_k,k};\bar{\lambda}^{T_k,k}, \bar \nu^{T_k,k}),(x^\star; \lambda,\nu)) \leq 2^{K-k-1}\eps$ if 
   $$T_k \geq \sqrt{\frac{24\Lada (D_x^{(k)})^2}{2^{K-k-1}\eps}}$$ 
   by  \cref{thm:heterogeneous-composite-convergence}.
   Hence, we just need to verify our choice of $T_k$ satisfies this inequality in each case. Then, to complete the induction, we establish
   \begin{equation} \label{eq:restarted-induction-step-two}
       Q((\bar x^{T_k,k};\bar{\lambda}^{T_k,k}, \bar \nu^{T_k,k}),(x^\star; \lambda,\nu)) \leq 2^{K-k-1}\eps \implies D_x^{(k+1)} \geq \|x^{k+1}-x^\star\|, \quad D_\lambda^{(k+1)} \geq \|\lambda^{k+1}-\lambda^\star\| \ . 
   \end{equation}
   The key result to this end is the growth condition from Lemma~\ref{lemma:Hetero-growth-x}, which guarantees that
   $$ G_x(\|\bar x^{T_k,k}-x^\star\|) + G_\lambda(\|\bar{\lambda}^{T_k,k}-\lambda^\star\|)  \leq  Q((\bar x^{T_k,k};\bar{\lambda}^{T_k,k}, \bar \nu^{T_k,k}),(x^\star; {\color{blue}\lambda^\star,\nabla g(\bar{x}^{T_k,k})})) \leq 2^{K-k-1}\eps \ . $$
   The remainder of this proof verifies the implications~\eqref{eq:restarted-induction-step-one} and~\eqref{eq:restarted-induction-step-two} and calculates the total gradient and proximal complexity in each case of Table~\ref{tab:composite-heterogeneous-rates}. Finally, we deduce that $$Q((\bar x^{T_{K-1},K-1};\bar \lambda^{T_{K-1},K-1},\bar \nu^{T_{K-1},K-1} ), (x^\star, \hat{\lambda}, \nabla g(\bar x^{T_{K-1},K-1}))) \leq \eps$$
   and apply \cref{lemma:Q-small->eps-opt} to conclude that $\bar x^{T_{K-1},K-1}$ is $(\eps,r)$-optimal.

   \paragraph{Case 1: } Suppose $\muada < 4\eps/D_x^2$. 
    Observe the first needed implication for our induction \eqref{eq:restarted-induction-step-one} is immediate from \cref{thm:heterogeneous-composite-convergence} as
    $$T_k = \sqrt{\frac{24\Lada D_x^2}{2^{K-k-1}\eps}}\ .$$
    The gradient complexity follows from geometrically summing this quantity {\color{blue} and bounding $K < \infty$}, so 
    $$ \sum_{k=0}^{K-1} \lceil T_k\rceil \leq \sum_{k=0}^{K-1} 1+ \sqrt{\frac{24\Lada D_x^2}{2^{K-k-1}\eps}} \leq K+  \sum_{j=0}^{\infty} \sqrt{\frac{24\Lada D_x^2}{2^{j}\eps}} = \sqrt{\frac{(144+96\sqrt{2})\Lada D_x^2}{\eps}}+K \ .$$
    Next, we verify the second needed implication~\eqref{eq:restarted-induction-step-two}. The primal bound is vacuously the case since the primal initialization is constant, so $x^k = x^0$ for each $k = 0, \dots, K-1$ and $$D_x^{(k)}=D_x \geq \|x^0-x^\star\|=\|x^k-x^\star\|\  . $$
    To derive the dual distance bound, we consider the two cases of dual restarting.
\begin{enumerate}
    \item[] \textbf{Case 1a}: Suppose $L_h > D_\lambda^2/\eps$.  In this setting, the dual variable does not reinitialize each iteration and $D_\lambda^{(k)}=D_\lambda \geq \|\lambda^0-\lambda^\star\|=\|\lambda^k-\lambda^\star\|$, completing the proof of~\eqref{eq:restarted-induction-step-two}. Observe that since $\Delta^{(k)}=D_\lambda/(2D_x\Lada)$ as neither variable reinitializes, the number of proximal steps on $u$ and $h$ taken each iteration $k$ of R-UFCM is \begin{align}\label{eq-prox-complexity-rufcm}\Peps^{(k)}=\lceil T_k \rceil+\lceil T_k\rceil ^2\Delta^{(k)}M &\leq 1+\sqrt{\frac{24\Lada D_x^2}{2^{K-k-1}\eps}}+\frac{24MD_xD_\lambda^{(k)}}{2^{K-k-1}\eps}+\frac{MD_\lambda^{(k)}}{\Lada D_x},\\
    &=1+\sqrt{\frac{24\Lada D_x^2}{2^{K-k-1}\eps}}+\frac{24MD_xD_\lambda}{2^{K-k-1}\eps}+\frac{MD_\lambda}{\Lada D_x} \nonumber\end{align}
    where $M$ bounds $\|\nabla g(x)\|$ for $x \in B(x^\star, \sqrt{{\color{blue}27}D_x^2})$, and we use the facts {\color{blue} $\lceil T_k\rceil \leq T_k + 1$ and $\lceil T_k \rceil^2 \leq 2T_k^2+2$}.
    The total proximal complexity is then at most $$\sum_{k=0}^{K-1}\Peps^{(k)} \leq \sqrt{\frac{(144+96\sqrt{2})\Lada D_x^2}{\eps}}+\frac{48MD_xD_\lambda}{\eps}+ K\left(1+\frac{MD_\lambda}{\Lada D_x}\right) \ . $$
    
    \item[] \textbf{Case 1b}: Suppose $L_h \leq D_\lambda^2/\eps$. To verify~\eqref{eq:restarted-induction-step-two}, {\color{blue} we note that for the primary executions of UFCM where $D_\lambda^2 < 2^{K-k+1}\eps L_h$, our initializations maintain $D_\lambda^{(k)}=D_\lambda$ and $\lambda^{k} = \lambda^0$. So $D_\lambda^{(k)} \geq \|\lambda^{k} - \lambda^\star\|$. For the subsequent executions,} observe that \cref{lemma:Hetero-growth-x} ensures $$G_\lambda(\|\lambda^{k+1}-\lambda^\star\|) \leq  Q((\bar x^{T_k,k};\bar{\lambda}^{T_k,k}, \bar \nu^{T_k,k}),(x^\star; {\color{blue}\lambda^\star,\nabla g(\bar{x}^{T_k,k})})) \leq 2^{K-k-1}\eps\ ,$$ where $\lambda^{k+1} \gets \bar{\lambda}^{T_k,k}$ by line 5 of \Cref{alg:R-FCM}. We then utilize the growth bound to yield
    $$\|\lambda^{k+1}-\lambda^\star\| \leq G_\lambda^{-1}(2^{K-k-1}\eps) = \sqrt{2^{K-k}\eps L_h} = D_\lambda^{(k+1)} \ ,$$ 
    completing our induction in this case.

    {\color{blue} From the proximal complexity for application $k$ of UFCM~\labelcref{eq-prox-complexity-rufcm}, we note that \begin{align*}\Peps^{(k)} &\leq 1 + \sqrt{\frac{24\Lada D_x^2}{2^{K-k-1}\eps}}+\frac{24MD_x\sqrt{2^{K-k+1}\eps L_h}}{2^{K-k-1}\eps}+\frac{MD_\lambda}{\Lada D_x}\\
    &\leq 1+\sqrt{\frac{(48\Lada + 1152M^2L_h)D_x^2}{2^{K-k-1}\eps}}+\frac{MD_\lambda}{\Lada D_x}\ , 
    \end{align*}
    since $D_\lambda^{(k)} \leq \min\{D_\lambda, \sqrt{2^{K-k+1}\eps L_h}\}$ for each $k$. The total proximal complexity is then bounded by $$\sum_{k=0}^{K-1} \Peps^{(k)} \leq  \sqrt{\frac{(6+4\sqrt{2})(48\Lada + 1152 M^2L_h)D_x^2}{\eps}}+K\left(1+\frac{MD_\lambda}{\Lada D_x}\right)\ .$$
    }
    \end{enumerate}

\paragraph{Case 2:} Now suppose $\muada \geq 4\eps/D_x^2$. Observe that the first step of our induction \labelcref{eq:restarted-induction-step-one} holds immediately after noting $D_x^{(k)}=\sqrt{2^{K-k+1}\eps/\muada}$ and applying \cref{thm:heterogeneous-composite-convergence} as $$T_k = \sqrt{\frac{96\Lada}{\muada}} = \sqrt{\frac{96 \Lada\left(D_x^{(k)}\right)^2}{2^{K-k+1}\eps}} \ . $$

  Hence, the total gradient complexity is bounded by $$\sum_{k=0}^{K-1} \lceil T_k\rceil  \leq K{\color{blue}\left(1+\sqrt{\frac{96\Lada}{\muada}}\right)} \ ,$$
{\color{blue} where we bound $\lceil T_k \rceil$ by $T_k + 1$. Note when $K$ is within a constant factor of $\left\lceil\log_2\left(\frac{Q(\Tilde z^0,\hat{z}^0)+\eps}{\eps}\right)\right\rceil$, then the gradient complexity is $\bigO\left(\sqrt{\frac{\Lada}{\muada}}\log\left(\frac{1}{\eps}\right)\right)$.}

Next, we verify the second needed implication \labelcref{eq:restarted-induction-step-two}, noting that the dual distance bounds $D_\lambda^{(k)} \geq \|\lambda^{k}-\lambda^\star\|$ have already been shown to hold. Thus, we only need to consider the primal distance bounds.  For $k = 0$, our initializations and \cref{lemma:UC-quad-lower-bound} ensure $$D_x^{(0)} = \sqrt{\frac{2^{K+1}\eps}{\muada}} \geq \sqrt{\frac{2(Q(\Tilde z^0,\hat{z}^0)+\eps)}{\muada}} \geq \|x^0-x^\star\| \ .$$ For $k \geq 1$, \cref{lemma:Hetero-growth-x} ensures that

$$G_x(\|x^{k+1}-x^\star\|) \leq  Q((\bar x^{T_k,k};\bar{\lambda}^{T_k,k}, \bar \nu^{T_k,k}),(x^\star; {\color{blue} \lambda^\star,\nabla g(\bar{x}^{T_k,k})})) \leq 2^{K-k-1}\eps \ ,$$
where $x^{k+1} \gets \bar{x}^{T_{k},k}$ by line 4 of \Cref{alg:R-FCM}.    
This bound implies $$\|{x}^{k+1}-x^{\star}\| \leq G_x^{-1}(2^{K-k-1}\eps)=\sqrt{2^{K-k}\eps/\mu_{2^{K-k}\eps}^{\mathtt{ADA}}}\leq \sqrt{2^{K-k}\eps/\muada} = D_x^{(k+1)} \ ,$$
where the first equality comes from the characterization that $\muada = \eps/(G_x^{-1}(\eps/2))^2$ and the last inequality holds as $\muada$ is a nondecreasing function of $\eps$. 
Next, we consider the proximal operator complexity.
\begin{enumerate}
    \item[] \textbf{Case 2a}: Suppose $L_h > D_\lambda^2/\eps$. In this case, $D_\lambda^{(k)}=D_\lambda$. The number of proximal steps performed each execution of UFCM is  $$\Peps^{(k)} = \lceil T_k \rceil+\lceil T_k\rceil ^2\Delta^{(k)}M \leq 1 + \sqrt{\frac{96\Lada}{\muada}}+\frac{192MD_\lambda}{\sqrt{2^{K-k+1} \muada \eps} } \ ,  $$
    where the first equality uses \labelcref{eq:general-P_eps}, and the inequality substitutes $\Delta^{(k)}=D_\lambda/(2\Lada D_x^{(k)})$ with $D_x^{(k)}=\sqrt{2^{K-k+1}\eps/\muada}$, further noting that since $\muada \leq 96\Lada$, $\lceil T_k \rceil \leq T_k+1 \leq 2T_k$. Therefore, the total proximal complexity is bounded by $$\sum_{k=0}^{K-1} \Peps^{(k)} \leq \frac{(1152+786\sqrt{2})MD_\lambda}{\sqrt{ \muada \eps} }+K{\color{blue}\left(1+\sqrt{\frac{96\Lada}{\muada}}\right)}\ . $$

    \item[] \textbf{Case 2b}: Suppose $L_h \leq D_\lambda^2/\eps$. Now $D_\lambda^{(k)} \leq \sqrt{2^{K-k+1} \eps L_h}$, in which case $$\Peps^{(k)} \leq 1+\sqrt{\frac{96\Lada}{\muada}}+\frac{192M\sqrt{L_h}}{\sqrt{\muada}} \leq 1+\sqrt{\frac{192\Lada+{\color{blue}73728}M^2 L_h}{\muada}} \ .$$
    The total proximal complexity is then bounded by $$\sum_{k=0}^{K-1} \Peps^{(k)} \leq K{\color{blue} \left(1+\sqrt{\frac{192\Lada+{\color{blue}73728}M^2 L_h}{\muada}}\right)}\ . $$
\end{enumerate}

\subsection{Application to Functionally Constrained Optimization}\label{section:application-to-func-constrained}

We conclude this section considering functionally constrained optimization with strongly convex and smooth components, recovering the linear convergence in terms of first-order oracle calls to $g$ and sublinear convergence in terms of proximal operations analogous to~\citep{zhang2022solving}.  In this setting, $h(z_0, \dots, z_m) = z_0+\iota_{z \leq 0}(z_1, \dots, z_m)$ with each $g_j$ being $\mu_j$-strongly convex results in constant $\muada = \mu_0 + \sum_{j=1}^m \lambda^\star_j\mu_j$. (Note that our method and theory also apply more generally, given only H\"older smoothness and uniform convexity, but for the sake of this comparison, we restrict ourselves to considering only smooth and strongly convex constraints.)

Since $h$ is nonsmooth, i.e.~{\color{blue} $L_h=\infty$}, each restarted application of UFCM uses the fixed dual initialization $\lambda^0$. However, for small enough $\eps$, the primal variables and distance bounds will update, with $D_x^{(k)} =  \sqrt{2^{K-k+1}{\eps}/{\muada}}$. Therefore, \Cref{alg:R-FCM} reaches an $\eps$-optimal solution with complexity bounds \begin{align*}
  \Neps = {\bigO}\left( \sqrt{\frac{L_{\eps,r}^{\mathtt{ADA}}}{\muada}}\log\left({\frac{Q(\Tilde z^0,\hat{z}^0)}{\eps}}\right)\right),  \quad
    \Peps = {\bigO}\left(\Neps+\frac{D_\lambda{M}}{\sqrt{\muada\eps}}\right) \ 
 \end{align*} 
In contrast, the ACGD-S method of~\citep[Corollary 4]{zhang2022solving} has oracle complexities $$\Neps = {\bigO}\left(\sqrt{\frac{L(\Lambda_r)}{\mu_0}}\log\left(\frac{\sqrt{{L(\Lambda_r)}{\mu_0}}D_x^2}{\eps}\right)\right), \quad \Peps = {\bigO}\left(\Neps + \frac{d(\Lambda_r)M}{\sqrt{\mu_0 \eps}}\right) \ ,$$
where $L(\Lambda_r) = \max_{\lambda \in \Lambda_r}\{\sum_{j=1}^m \lambda_jL_j\}$ and $d(\Lambda_r)=\|\lambda^\star\|+r$.
In the case where only $g_0$ is strongly convex, our rate recovers theirs as $\muada = \mu_0 + \sum_{j=1}^m \lambda^\star_j 0 = \mu_0$. Importantly, our method additionally benefits from strong convexity in the components as $\muada > \mu_0$ whenever any active constraint is strongly convex (or even just, uniformly convex).

\section*{Funding}
This work was supported in part by the Air Force Office of Scientific
Research under award number FA9550-23-1-0531. Benjamin Grimmer was additionally supported
as a fellow of the Alfred P. Sloan Foundation


\appendix
\section{Deferred Proofs}

\subsection{Deferred Proofs for Smooth Composite Analysis}\label{section:appendix-smooths-comp}

\subsubsection*{\textit{Proof of \cref{lemma:x-lambda-convergence}}.} 

       First we establish a convergence bound on the inner loop for each phase. Fix $t \geq 1$. Since $u(y)+\eta_t\|y-x^{t-1}\|^2/2$ has strong convexity with modulus $\eta_t$, the proximal step for $y$ in line 8 of \Cref{alg:FCM} satisfies the three point inequality (see~\citep[Lemma 3.5]{Lan.G})

    \begin{align}\label{eq:initial-y-prox-bound}
    \begin{split}
        \inner{y_s^{(t)}-x}{\tilde{h}^{(t),s}}&+u(y_s^{(t)})-u(x)+\frac{\eta_t}{2}\left(\|y_s^{(t)}-x^{t-1}\|^2-\|x-x^{t-1}\|^2\right)\\
        &+\frac{1}{2}\left[(\beta^{(t)}+\eta_t)\|x-y_s^{(t)}\|^2+\beta^{(t)}\|y_s^{(t)}-y_{s-1}^{(t)}\|^2-\beta^{(t)}\|y_{s-1}^{(t)}-x\|^2\right] \leq 0 \ .
        \end{split}
    \end{align}
    Recall that $$\Tilde{h}^{(t),s}=\begin{cases}
  (\nu^t)^\top\lambda_0^{(t)}+\rho^{(t)}(\nu^{t-1})^\top(\lambda_0^{(t)}-\lambda_{-1}^{(t)}) & \text{if} \ s=1,\\
       (\nu^t)^\top\lambda_{s-1}^{(t)}+(\nu^{t})^\top(\lambda_{s-1}^{(t)}-\lambda_{s-2}^{(t)}) & \text{otherwise.}\\
       \end{cases}$$  
       In particular, line 7 of \Cref{alg:FCM} ensures that for $s \geq 2$
    \begin{align}\label{eq:proxy-expansion-for-initial-composite-bound}
         \inner{y_s^{(t)}-x}{\tilde{h}^{(t),s}}& = \inner{y_s^{(t)}-x}{ \sum_{j=1}^m \lambda_{s,j}^{(t)}\nu_j^t }- \inner{y_s^{(t)}-x}{ \sum_{j=1}^m (\lambda_{s,j}^{(t)}-\lambda_{s-1,j}^{(t)})\nu_j^t }\\
         &+\inner{y_s^{(t)}-y_{s-1}^{(t)}}{\sum_{j=1}^m(\lambda_{s-1,j}^{(t)}-\lambda_{s-2,j}^{(t)})\nu_j^t}+\inner{y_{s-1}^{(t)}-x}{\sum_{j=1}^m(\lambda_{s-1,j}^{(t)}-\lambda_{s-2,j}^{(t)})\nu_j^t}\ , \nonumber
    \end{align}
   and for $s = 1$ 
       \begin{align}\label{eq:proxy-expansion-for-initial-composite-bound-s=1}
         \inner{y_1^{(t)}-x}{\tilde{h}^{(t),1}}& = \inner{y_1^{(t)}-x}{ \sum_{j=1}^m \lambda_{1,j}^{(t)}\nu_j^t }- \inner{y_1^{(t)}-x}{ \sum_{j=1}^m (\lambda_{1,j}^{(t)}-\lambda_{0,j}^{(t)})\nu_j^t }\\
         &+\rho^{(t)}\inner{y_1^{(t)}-y_{0}^{(t)}}{\sum_{j=1}^m(\lambda_{0,j}^{(t)}-\lambda_{-1,j}^{(t)})\nu_j^{t-1}}+\rho^{(t)}\inner{y_{0}^{(t)}-x}{\sum_{j=1}^m(\lambda_{0,j}^{(t)}-\lambda_{-1,j}^{(t)})\nu_j^{t-1}} \ . \nonumber
    \end{align}
Observe the third term is bounded by Young's inequality $(ab \leq \frac{a^2}{2\eps}+\frac{b^2\eps}{2} \text{ for all } \eps>0)$ with \begin{align}\label{eq:y-iterate-young-inequality}
\begin{split}
    \inner{y_s^{(t)}-y_{s-1}^{(t)}}{\sum_{j=1}^m (\lambda_{s-1,j}^{(t)}-\lambda_{s-2,j}^{(t)})\nu_j^t} \leq \frac{\beta^{(t)}\|y_s^{(t)}-y_{s-1}^{(t)}\|^2}{2}+\frac{\|\lambda_{s-1}^{(t)}-\lambda_{s-2}^{(t)}\|^2\|\nu^t\|^2}{2\beta^{(t)}} \ ,\\
 {\color{blue} \inner{y_1^{(t)}-y_{0}^{(t)}}{\sum_{j=1}^m (\lambda_{0,j}^{(t)}-\lambda_{-1,j}^{(t)})\nu_j^{t-1}} \leq \frac{\beta^{(t)}\|y_1^{(t)}-y_{0}^{(t)}\|^2}{2\rho^{(t)}}+\frac{\rho^{(t)}\|\lambda_{0}^{(t)}-\lambda_{-1}^{(t)}\|^2\|\nu^{t-1}\|^2}{2\beta^{(t)}} \ ,}
    \end{split}
\end{align}
{\color{blue}for $s \geq 2$ and $s = 1$ respectively.} {\color{blue} Moreover, note that when summing over $s = 1, ..., S_t$, \begin{align*}
     \sum_{s=1}^{S_t} \Lcal(y_{s}^{(t)}; \lambda_s^{(t)}, \nu^t) - \Lcal(x; \lambda_s^{(t)}, \nu^t) &=  \sum_{s=1}^{S_t} \inner{y_s^{(t)}-x}{\sum_{j=1}^m \lambda_{s,j}^{(t)}\nu_j^t}+u(y_s^{(t)})-u(x), \\
    \frac{\beta^{(t)}}{2}\left(\|y_{S_t}^{(t)}-x\|^2-\|y_{0}^{(t)}-x\|^2\right) &= \sum_{s=1}^{S_t} \frac{\beta^{(t)}}{2}\left(\|x-y_s^{(t)}\|^2-\|y_{s-1}^{(t)}-x\|^2\right),\\
    \sum_{s=2}^{S_t} \frac{\|\lambda_{s-1}^{(t)}-\lambda_{s-2}^{(t)}\|^2\|\nu^t\|^2}{2\beta^{(t)}} & \leq\displaystyle\sum_{s=2}^{S_t} \frac{\gamma^{(t)}}{2}\|\lambda_{s-1}^{(t)}-\lambda_{s-2}^{(t)}\|^2,
   \end{align*} 
   where the first equality holds by definition, the second holds from telescoping the norm squared terms, and the inequality holds from requirement \labelcref{cond:b3}. Furthermore, the inner product terms telescope as well since
   \begin{align*}
           \rho^{(t)}\inner{y_0^{(t)}-x}{\sum_{j=1}^m(\lambda_{0,j}^{(t)}-\lambda_{-1,j}^{(t)})\nu_j^{t-1}}&-\inner{y_{S_t}^{(t)}-x}{\sum_{j=1}^m(\lambda_{S_t,j}^{(t)}-\lambda_{S_t-1,j}^{(t)})\nu_j^{t}} \\
    &\hspace{-1cm}=\sum_{s=2}^{S_t} \inner{y_{s-1}^{(t)}-x}{\sum_{j=1}^m(\lambda_{s-1,j}^{(t)}-\lambda_{s-2,j}^{(t)})\nu_j^t} - \inner{y_s^{(t)}-x}{ \sum_{j=1}^m (\lambda_{s,j}^{(t)}-\lambda_{s-1,j}^{(t)})\nu_j^t }\\
    &\hspace{-0.75cm}+ \rho^{(t)}\inner{y_{0}^{(t)}-x}{\sum_{j=1}^m(\lambda_{0,j}^{(t)}-\lambda_{-1,j}^{(t)})\nu_j^{t-1}} -  \inner{y_1^{(t)}-x}{ \sum_{j=1}^m (\lambda_{1,j}^{(t)}-\lambda_{0,j}^{(t)})\nu_j^t },
   \end{align*} where the substitutions $\lambda_0^{(t+1)} = \lambda_{S_t}^{(t)}, \ \lambda_{-1}^{(t+1)}=\lambda_{S_t-1}^{(t)}, \ y_0^{(t+1)}=y_{S_t}^{(t)}$ hold by line 11 of \Cref{alg:FCM}.}

We then sum \labelcref{eq:initial-y-prox-bound} over $s = 1, \dots, S_t$. After plugging in \labelcref{eq:proxy-expansion-for-initial-composite-bound} and \labelcref{eq:proxy-expansion-for-initial-composite-bound-s=1}, applying the above {\color{blue} equalities and inequalities}, and considering requirement \labelcref{cond:b3}, we bound
   \begin{align}\label{eq:y-gap-bound}
      \sum_{s=1}^{S_t}& \left(\Lcal(y_s^{(t)}; \lambda_s^{(t)}, \nu^t)-\Lcal(x; \lambda_s^{(t)},  \nu^t)+\frac{\eta_t\|y_s^{(t)}-x^{t-1}\|^2+\eta_t\|y_s^{(t)}-x\|^2-\eta_t\|x-x^{t-1}\|^2}{2}\right) \nonumber\\
      &+\rho^{(t)}\inner{y_0^{(t)}-x}{\sum_{j=1}^m(\lambda_{0,j}^{(t)}-\lambda_{-1,j}^{(t)})\nu_j^{t-1}}-\inner{y_{S_t}^{(t)}-x}{\sum_{j=1}^m(\lambda_{S_t,j}^{(t)}-\lambda_{S_t-1,j}^{(t)})\nu_j^{t}}\\
      &\leq\displaystyle\sum_{s=2}^{S_t} \frac{\gamma^{(t)}}{2}\|\lambda_{s-1}^{(t)}-\lambda_{s-2}^{(t)}\|^2+\frac{{\color{blue}(\rho^{(t)})^2}\|\nu^{t-1}\|^2}{2\beta^{(t)}}\|\lambda_0^{(t)}-\lambda_{-1}^{(t)}\|^2
      -\frac{1}{2}\left[\beta^{(t)}\|y_{S_t}^{(t)}-x\|^2-\beta^{(t)}\|y_0^{(t)}-x\|^2\right] , \nonumber
   \end{align}
    since the terms $\|y_s^{(t)}-y_{s-1}^{(t)}\|^2$ cancel as \labelcref{eq:initial-y-prox-bound} and \labelcref{eq:y-iterate-young-inequality} have the same coefficients. 
   
Next we leverage this inner loop bound to derive bounds on the $\lambda_s^{(t)}$ terms. The proximal mapping in Line 9 of \Cref{alg:FCM} applies a proximal step to $$p(\lambda) := \inner{\lambda}{\nu^t(y_s^{(t)}-\underline{x}^t)+g(\underline{x}^t)}-h^*(\lambda) \ .$$
   From the Fenchel-Young inequality,  $\inner{\lambda}{ \nu^t y_s^{(t)}-g^*(\nu^t)} = \inner{\lambda}{\nu^t(y_s^{t}-\underline{x}^t)+g(\underline{x}^t)}$. Therefore, the proximal mapping in line 9, the three point inequality of~\citep[Lemma 3.5]{Lan.G}, and $L_h$-smoothness of $h$ imply the following is nonpositive
   $$\hLcal(y_s^{(t)}; \lambda,  \nu^t)-\hLcal(y_s^{(t)}; \lambda_s^{(t)}, \nu^t) + \frac{\gamma^{(t)}}{2}\left[\|\lambda-\lambda_s^{(t)}\|^2+\|\lambda_s^{(t)}-\lambda_{s-1}^{(t)}\|^2-\|\lambda-\lambda_{s-1}^{(t)}\|^2\right] + \muhstar\|\lambda_s^{(t)}-\lambda\|^2 \ .$$  
 Taking the sum over $s = 1, \dots, S_t$ and combining with \labelcref{eq:y-gap-bound},
   \begin{align*}
       &\sum_{s=1}^{S_t}  \left(\hLcal(y_s^{(t)}; \lambda,  \nu^t)-\hLcal(x; \lambda_s^{(t)},  \nu^t)+\frac{\eta_t\|y_s^{(t)}-x^{t-1}\|^2+\eta_t\|y_s^{(t)}-x\|^2-\eta_t\|x-x^{t-1}\|^2}{2}\right)\\&\hspace{0.2cm}+\sum_{s=1}^{S_t}\muhstar\|\lambda_s^{(t)}-\lambda\|^2+\rho^{(t)}\inner{y_0^{(t)}-x}{\sum_{j=1}^m(\lambda_{0,j}^{(t)}-\lambda_{-1,j}^{(t)})\nu_j^{t-1}}-\inner{y_{S_t}^{(t)}-x}{\sum_{j=1}^m(\lambda_{S_t,j}^{(t)}-\lambda_{S_{t}-1,j}^{(t)})\nu_j^{t}}\\
      &\hspace{0.2cm}\leq \frac{\gamma^{(t)}}{2}\|\lambda-\lambda_0^{(t)}\|^2-\frac{\gamma^{(t)}}{2}\left[\|\lambda-\lambda_{S_t}^{(t)}\|^2+\|\lambda_{S_t}^{(t)}-\lambda_{S_t-1}^{(t)}\|^2\right]+\frac{{\color{blue}(\rho^{(t)})^2}\|\nu^{t-1}\|^2}{2\beta^{(t)}}\|\lambda_0^{(t)}-\lambda_{-1}^{(t)}\|^2
      \\&\hspace{1.5cm}-\frac{\beta^{(t)}}{2}\left[\|y_{S_t}^{(t)}-x\|^2-\|y_0^{(t)}-x\|^2\right] \ ,
    \end{align*}
   {\color{blue} noting that the $\|\lambda-\lambda_s^{(t)}\|^2$ terms telescope and $\gamma^{(t)}/2\sum_{s=2}^{S_t}\|\lambda_{s-1}^{(t)}-\lambda_{s-2}^{(t)}\|^2$ cancels, leaving only $\gamma^{(t)}/2\|\lambda_{S_t}^{(t)}-\lambda_{{S_t}-1}\|^2$.}
    Since $\hLcal(y_s^{(t)}; \lambda,  \nu^t)$ and $\|y_s^{(t)}-x\|^2$ are convex in $y_s^{(t)}$ and $\Lcal(x;\lambda_s^{(t)},\nu^t)$ is {\color{blue} concave} in $\lambda_s^{(t)}$, multiplying by $\Tilde{\omega}^{(t)}=\omega_t/S_t$ and considering the averaging scheme in line 14 of \Cref{alg:FCM}, one can apply Jensen's inequality to derive a bound with respect to $x^t$ and $\Tilde{\lambda}^t$ of
    \begin{align}
        &\omega_t\left(\hLcal(x^t; \lambda,  \nu^t)-\hLcal(x; \Tilde{\lambda}^t, \nu^t)+\frac{\eta_t}{2}\left({\|x^t-x^{t-1}\|^2+\|x^t-x\|^2-\|x^{t-1}-x\|^2}\right)\right)+\sum_{s=1}^{S_t}\frac{\omega_t\|\lambda-\lambda_s^{(t)}\|^2}{2L_hS_t}\nonumber\\
        &+\Tilde{\omega}^{(t)}\rho^{(t)}\inner{y_0^{(t)}-x}{\sum_{j=1}^m(\lambda_{0,j}^{(t)}-\lambda_{-1,j}^{(t)})\nu_j^{t-1}}-\Tilde{\omega}^{(t)}\inner{y_{S_t}^{(t)}-x}{\sum_{j=1}^m(\lambda_{S_t,j}^{(t)}-\lambda_{S_t-1,j}^{(t)})\nu_j^{t}}\nonumber \\
      &\leq \frac{\Tilde{\omega}^{(t)}\gamma^{(t)}}{2}\left(\|\lambda-\lambda_0^{(t)}\|^2-\|\lambda-\lambda_{S_t}^{(t)}\|^2-\|\lambda_{S_t}^{(t)}-\lambda_{S_t-1}^{(t)}\|^2\right)\nonumber \\
      &+\frac{\Tilde{\omega}^{(t)}{\color{blue}(\rho^{(t)})^2}\|\nu^{t-1}\|^2}{2\beta^{(t)}}\|\lambda_0^{(t)}-\lambda_{-1}^{(t)}\|^2-\frac{\Tilde{\omega}^{(t)}\beta^{(t)}}{2}\left[\|y_{S_t}^{(t)}-x\|^2-\|y_0^{(t)}-x\|^2\right] \ .
    \end{align}
    Next, we note that {\color{blue} \begin{align*}
      \hQ_x(z^t, z)+\hQ_\lambda(z^t,z) &=  \hLcal(x^t; \lambda,  \nu^t)-\hLcal(x; \Tilde{\lambda}^t,  \nu^t),\\
     \frac{\omega_T\eta_T}{2}\|x^T-x\|^2-\frac{\omega_1\eta_1}{2}\|x_0-x\|^2&\leq \sum_{t=1}^T \frac{\omega_t \eta_t}{2}\left(\|x^t-x\|^2-\|x^{t-1}-x\|^2\right)
    \end{align*}
       where the first equality comes from definition and the inequality comes from telescoping along with requirement \labelcref{cond:a1}. 
      Then by telescoping, we produce the following bounds \begin{align*}
            \sum_{t=1}^T \frac{\Tilde{\omega}^{(t)}\beta^{(t)}}{2}\left(\|y_0^{(t)}-x\|^2-\|y_{S_t}^{(t)}-x\|^2\right) &\leq \frac{\Tilde{\omega}^{(1)}\beta^{(1)}}{2}\|y_0^{(1)}-x\|^2 - {\color{blue}\frac{\Tilde{\omega}^{(T)}\beta^{(T)}}{2}\|y_{S_T}^{(T)}-x\|^2 }\\
       \sum_{t=1}^T \frac{\Tilde{\omega}^{(t)}\gamma^{(t)}}{2}\left(\|\lambda-\lambda_0^{(t)}\|^2-\|\lambda-\lambda_{S_t}^{(t)}\|^2\right) &\leq \frac{\Tilde{\omega}^{(1)}\gamma^{(1)}}{2}\|\lambda-\lambda_0^{(1)}\|^2 - {\color{blue}\frac{\Tilde{\omega}^{(T)}\gamma^{(T)}}{2}\|\lambda-\lambda_{S_T}^{(T)}\|^2 }\\
        &\hspace{-6cm}\sum_{t=1}^T -\frac{\Tilde{\omega}^{(t)}\gamma^{(t)}}{2}\|\lambda_{S_t}^{(t)}-\lambda_{S_t-1}^{(t)}\|^2+\frac{\Tilde{\omega}^{(t)}(\rho^{(t)})^2\|\nu^{t-1}\|^2}{2\beta^{(t)}}\|\lambda_0^{(t)}-\lambda_{-1}^{(t)}\|^2 \\&\leq \sum_{t=1}^T \frac{\Tilde{\omega}^{(t)}\gamma^{(t)}}{2}\left(\|\lambda_0^{(t)}-\lambda_{-1}^{(t)}\|^2-\|\lambda_{S_t}^{(t)}-\lambda_{S_t-1}^{(t)}\|^2\right)\\
           &\leq \frac{\Tilde{\omega}^{(1)}\gamma^{(1)}}{2}\|\lambda_0^{(1)}-\lambda_{-1}^{(1)}\|^2 - {\color{blue}\frac{\Tilde{\omega}^{(T)}\gamma^{(T)}}{2}\|\lambda_{S_T}^{(T)}-\lambda_{S_T-1}^{(T)}\|^2 }
\end{align*}
where the first two inequalities apply \labelcref{cond:c1} and \labelcref{cond:c2} respectively, while the following bound applies requirement \labelcref{cond:c3},   line 11 of Algorithm~\ref{alg:FCM}, and uses \labelcref{cond:c2} to telescope.

Finally for $t \geq 2$, by line 11 of Algorithm~\ref{alg:FCM} and the second condition of \labelcref{cond:c3},
    \begin{align*}
        \Tilde{\omega}^{(t)}\rho^{(t)}\inner{y_0^{(t)}-x}{\sum_{j=1}^m(\lambda_{0,j}^{(t)}-\lambda_{-1,j}^{(t)})\nu_j^{t-1}}-\Tilde{\omega}^{(t)}\inner{y_{S_t}^{(t)}-x}{\sum_{j=1}^m(\lambda_{S_t,j}^{(t)}-\lambda_{S_t-1,j}^{(t)})\nu_j^{t}} &=\\
        \Tilde{\omega}^{(t-1)}\inner{y_{S_{t-1}}^{(t-1)}-x}{\sum_{j=1}^m(\lambda_{S_{t-1},j}^{(t-1)}-\lambda_{S_{t-1}-1,j}^{(t-1)})\nu_j^{t-1}}-\Tilde{\omega}^{(t)}\inner{y_{S_t}^{(t)}-x}{\sum_{j=1}^m(\lambda_{S_t,j}^{(t)}-\lambda_{S_t-1,j}^{(t)})\nu_j^{t}}.
    \end{align*} Since $\lambda_{-1}^{(1)}=\lambda_{0}^1$ by initialization, summing the inner product terms over $t= 1, ..., T$ results in $-\Tilde{\omega}^{(T)}\inner{y_{S_T}^{(T)}-x}{\sum_{j=1}^m (\lambda_{S_T,j}^{(T)}-\lambda_{S_T-1,j}^{(T)})\nu_j^T}$. Rearranging, applying Young's inequality, using requirement \labelcref{cond:b3}, and combining with the bounds outlined above results in}
    \begin{align}
    \sum_{t=1}^T \omega_t [\hQ_x(z^t,z)+\hQ_\lambda(z^t,z)]&+\sum_{t=1}^T\sum_{s=1}^{S_t}\frac{\omega_t}{S_t}\muhstar\|\lambda-\lambda_s^{(t)}\|^2+\sum_{t=1}^T \frac{\omega_t\eta_t}{2} \|x^t-x^{t-1}\|^2\nonumber \\
    &\hspace{-2cm}+\frac{\omega_T\eta_T}{2}\|x^T-x\|^2-\frac{\omega_1\eta_1}{2}\|x^0-x\|^2 \leq \frac{\Tilde{\omega}^{(1)}}{2}\left(\gamma^{(1)}\|\lambda_0^{(1)}-\lambda\|^2+\beta^{(1)}\|y_0^{(1)}-x\|^2\right) \ . \nonumber
\end{align}


\subsubsection*{\textit{Proof of \cref{lemma:nu-convergence-bound}}.}

     Recall that $\Tilde{x}^t=x^{t-1}+\theta_t(x^{t-1}-x^{t-2})$. Thus,
        \begin{align*}
            \inner{\nu_j-\nu_j^t}{(\Tilde{x}^t-x^t)} 
            &=-\inner{\nu_j-\nu_j^t}{(x^t-x^{t-1})}+\theta_t\inner{\nu_j-\nu_j^{t-1}}{(x^{t-1}-x^{t-2})}\\&+\theta_t\inner{\nu_j^{t-1}-\nu_j^{t}}{(x^{t-1}-x^{t-2})} \ .
        \end{align*}
        Using~\citep[Lemma 3.5]{Lan.G} and that $g_j^*$ is strongly convex with modulus 1 with respect to the Bregman divergence $U_{g_j^*}$, the proximal mapping $$ \nu_j^t \leftarrow \argmax_{\nu_j \in V_j} \inner{\nu_j}{\tilde{x}^t}-g_j^*(\nu_j)-\tau_tU_{g_j^*}(\nu;\nu^{t-1}) \quad \forall j \in \{1, \dots, m\} \ ,$$
       which is equivalent to line 4 of \Cref{alg:FCM}, satisfies $j \in \{1, \dots, m\}$ \begin{align*}
            \inner{\nu_j-\nu_j^t}{x^t}+&\inner{\nu_j-\nu_j^t}{\Tilde{x}^t-x^t}+g_{j}^*(\nu_j^t)-g_j^*(\nu_j)\\
            &\leq \tau_tU_{g_j^*}(\nu_j;\nu_j^{t-1})-(\tau_t+1)U_{g_j^*}(\nu_j;\nu_j^t)-\tau_tU_{g_j^*}(\nu_j^t;\nu_j^{t-1}) \ .
        \end{align*}
        Summing over $t = 1, \dots, T$ with weights $\omega_t$ yields
         \begin{align*}
            &\sum_{t=1}^T \omega_t\left[\inner{\nu_j-\nu_j^t}{x^t}+g_{j}^*(\nu_j^t)-g_j^*(\nu_j)\right]\\
            &+\sum_{t=1}^T -\omega_t\inner{\nu_j-\nu_j^t}{(x^t-x^{t-1})}+\sum_{t=1}^T\omega_t\theta_t\inner{\nu_j-\nu_j^{t-1}}{(x^{t-1}-x^{t-2})}\\&+\sum_{t=1}^T\omega_t\theta_t\inner{\nu_j^{t-1}-\nu_j^{t}}{(x^{t-1}-x^{t-2})}\\
            &\leq \sum_{t=1}^T \omega_t\left[\tau_tU_{g_j^*}(\nu_j;\nu_j^{t-1})-(\tau_t+1)U_{g_j^*}(\nu_j;\nu_j^t)-\tau_tU_{g_j^*}(\nu_j^t;\nu_j^{t-1})\right] \ .
        \end{align*}
        Applying requirements \labelcref{cond:a2} and \labelcref{cond:a3}, one can conclude that
        \begin{align*}
            \sum_{t=1}^T \omega_t\left[\inner{\nu_j-\nu_j^t}{x^t}+g_{j}^*(\nu_j^t)-g_j^*(\nu_j)\right]&\leq -\left[\omega_T(\tau_T+1)U_{g_j^*}(\nu_j;\nu_j^T)-\omega_T\inner{\nu_j-\nu_j^T}{x^T-x^{T-1}}\right]\\
            &\hspace{-3cm}-\left[\sum_{t=1}^T\omega_t\tau_tU_{g_j^*}(\nu_j^t;\nu_j^{t-1})-\omega_{t-1}\inner{\nu_j^{t-1}-\nu_j^{t}}{(x^{t-1}-x^{t-2})}\right]+\omega_1\tau_1U_{g_j^*}(\nu_j,\nu_j^0) \ .
        \end{align*}
        Taking the sum over $j = 1, \dots, m$ with weights $\lambda_j$ yields the desired result.

\subsubsection*{\textit{Proof of \cref{prop:Q-convergence-bound}}.}
    From \cref{lemma:zhang-smoothness-growth-lemma}, any $\nu \in V$ satisfies
        $$\sum_{j=1}^m \lambda_jU_{g_{j}^*}(\nu;\nu_j^{T}) \geq \frac{1}{2\Lada}\left\|\sum_{j=1}^m \lambda_j(\nu_j-\nu_j^{T})\right\|^2 \ .$$
        Therefore, by \cref{lemma:nu-convergence-bound} and requirement \labelcref{cond:a3},
           \begin{align}\label{eq:Q_nu-cocoersive_bregman_bound_appendix}
           \begin{split}
              \sum_{t=1}^T \omega_t\left[\hQ_\nu(z^t,z)\right]&\leq \left[ \omega_T\inner{\sum_{j=1}^m \lambda_j(\nu_j-\nu_j^T)}{x^T-x^{T-1}}- \frac{\omega_T(\tau_T+1)}{2\Lada}\left\|\sum_{j=1}^m \lambda_j(\nu_j-\nu_j^{T})\right\|^2\right]\\         &\hspace{-2cm}+\sum_{t=2}^T\left[\omega_{t-1}\inner{\sum_{j=1}^m \lambda_j(\nu_j^{t-1}-\nu_j^{t})}{(x^{t-1}-x^{t-2})}-\frac{\omega_{t-1}\tau_t}{2\theta_t\Lada}\left\|\sum_{j=1}^m \lambda_j(\nu_j^t-\nu_j^{t-1})\right\|^2\right]\\&\hspace{-2cm}+\omega_1\tau_1\left(\sum_{j=1}^m \lambda_jU_{g_{j}^*}(\nu_j,\nu_j^0)\right) \ .
                 \end{split}
        \end{align}
        Applying Young's inequality to the inner product yields
        \begin{align}\label{eq:Q_nu-bound-Young's-inequality}
              \sum_{t=1}^T \omega_t\left[\hQ_\nu(z^t,z)\right]&\leq \frac{\omega_T\Lada}{2(\tau_T+1)}\|x^T-x^{T-1}\|^2+\sum_{t=1}^{T-1} \frac{\omega_t\theta_{t+1}\Lada}{2\tau_{t+1}}\|x^t-x^{t-1}\|^2+\omega_1\tau_1\inner{\lambda}{U_{g^*}(\nu_j,\nu_j^0)} \ .
        \end{align}
       Utilizing the stepsize conditions \labelcref{cond:a3} and \labelcref{cond:a4}, combining with \cref{lemma:x-lambda-convergence}, and the fact that $y_0^{(1)}=x^0$ and $\lambda_0^{(1)}=\lambda^0$, achieves the desired bound.

\subsection{Proofs Deferred for Heterogeneously Smooth Composite Analysis}\label{section:appendix-heterogeneous-comp-proofs}

{\color{blue}\subsubsection*{\textit{Proof of \cref{prop:compactness of iterates general}}.}
Similar to the proof of \cref{lemma:h-smooth-lambda-iterate-bound}, we note the bound from \cref{prop:general-Holder-Q-bound}. Considering $\tau_1 = 0$ and the nonnegativity of $L_h$, the associated norm, and $Q(z^t,z^\star)$, it holds for all $T \geq 1$, \begin{align*}
   \frac{\omega_T\eta_T}{2}\|x^T-x^\star\|^2 &\leq \frac{\Tilde{\omega}^{(1)}\beta^{(1)}+\omega_1\eta_1}{2}\|x^0-x^\star\|^2+\frac{\Tilde{\omega}^{(1)}\gamma^{(1)}}{2}\|\lambda^0-\lambda^\star\|^2 +\frac{\delta}{2}\left[\omega_T(\tau_T+1)+\sum_{t=2}^T \omega_t\tau_t\right].
\end{align*}
 Using the stepsize conditions \labelcref{eq:outerloop-stepsizes} and \labelcref{eq:innerloop-stepsizes} and our distance bounds, and letting $N_\eps = \sqrt{\frac{24 \Lada D_x^2}{\eps}}$ then \begin{align*}
      \|x^T-x^\star\|^2 &\leq \frac{1}{2\Lada}\left[(C/\Delta + 2\Lada)D_x^2 + \frac{1}{C\Delta}D_\lambda^2+ \frac{12\Lada D_x^2}{N_\eps^3}\left(\frac{T^3+3T^2+2T}{6}\right) \right], \quad \forall T \geq 1, 
 \end{align*}
where we bounded $\left[\omega_T(\tau_T+1)+\sum_{t=2}^T \omega_t\tau_t\right] \leq \frac{T^3+3T^2+2T}{6}$ with specialized $\delta = \eps/N_\eps$. Finally considering particular iterate $t \geq 1$, we can further bound the rightmost product by $48\Lada D_x^2$ for any $t \leq \lceil N_\eps \rceil$.

The remainder of the proof follows analogously to the proof of \cref{lemma:h-smooth-lambda-iterate-bound}.}

\subsubsection*{\textit{Proof of \cref{lemma:Holder-smoothness-growth}}.}

    Let $\lambda \in \Lambda_r$ and $\bar{g}(x)=\sum_{j=1}^m \lambda_jg_j$.
    Moreover, consider any $\delta >0$ and $$L_\lambda \geq \sum_{j=1}^m \left[\left[\frac{1-p_j}{1+p_j}\cdot\frac{m}{\delta}\right]^\frac{1-p_j}{1+p_j}\lambda_j^\frac{2}{1+p_j}L_j^\frac{2}{1+p_j}\right] \ . $$
    Since $\nabla \bar{g}(y)=\sum_{j=1}^m \lambda_j \nabla g_j(y)$ for all $y \in X$, letting $\nu = \nabla g(x)$ and $\hat{\nu}=\nabla g(\hat{x})$, one has that
    \begin{align*}
        \inner{\lambda}{U_{g^*}(\nu;\hat{\nu})}
        = \sum_{j=1}^m \lambda_jU_{g_j^*}(\nu_j;\hat{\nu}_j) 
        = \sum_{j=1}^m \lambda_j U_{g_j}(\hat{x};x)
        = U_{\bar{g}}(\hat{x};x) 
        \geq \frac{1}{2L_\lambda}\|\sum_{j=1}^m \lambda_j(\nu_j-\hat{\nu}_j)\|^2-\frac{\delta}{2}
    \end{align*}
    where the second equality follows from~\citep[Appendix A.2]{bregman_duality_paper} and the inequality from using the cocoercivity condition in Lemma~\ref{lemma:cocoercive-Holder}.     Noting in \labelcref{eq:general-heterogeneous-smoothness-constant}, $L_{\delta,r} \geq L_\lambda$ for all $\lambda \in \Lambda_r$ gives the desired bound.

\subsubsection*{\textit{Proof of \cref{lemma:nu-convergence-Holder}}.}  
    This result follows analogously {\color{blue} to deriving~\eqref{eq:Q_nu-cocoersive_bregman_bound_appendix}} in the proof of \cref{prop:Q-convergence-bound}, with the modification of noting that by \cref{lemma:Holder-smoothness-growth}, for any $\nu \in V$ and $\delta>0$,
        $$\sum_{j=1}^m \lambda_jU_{g_j^*}(\nu;\nu_j^{T}) \geq \frac{1}{2L_{\delta,r}}\left\|\sum_{j=1}^m \lambda_j(\nu_j-\nu_j^{T})\right\|^2-\frac{\delta}{2} \ .$$\\
      By \cref{lemma:nu-convergence-bound} and requirement  \labelcref{cond:a3}, with the above substitution
          \begin{align*}
              \sum_{t=1}^T \omega_t\left[\hQ_\nu(z^t,z)\right] &\leq \omega_T(\tau_T+1)\left[\frac{\delta}{2}- \frac{1}{2L_{\delta,r}}\left\|\sum_{j=1}^m \lambda_j(\nu_j-\nu_j^{T})\right\|^2\right] \\&\hspace{-0.5cm}+\omega_T\inner{\sum_{j=1}^m \lambda_j(\nu_j-\nu_j^T)}{x^T-x^{T-1}}      +\sum_{t=2}^T\biggl[\omega_t\tau_t\left(\frac{\delta}{2}-\frac{1}{2L_{\delta,r}}\left\|\sum_{j=1}^m \lambda_j(\nu_j^t-\nu_j^{t-1})\right\|^2\right)\\&\hspace{-0.5cm}+\omega_{t-1}\inner{\sum_{j=1}^m \lambda_j(\nu_j^{t-1}-\nu_j^{t})}{(x^{t-1}-x^{t-2})}\biggr]+\omega_1\tau_1\left(\sum_{j=1}^m \lambda_jU_{g_j^*}(\nu_j,\nu_j^0)\right) \ .
        \end{align*} 
        Rearranging and applying Young's inequality (analogous to~\eqref{eq:Q_nu-bound-Young's-inequality}) concludes the proof.

\end{document}